\documentclass[11 pt]{amsart}

\usepackage{amsmath}
\usepackage{amssymb}
\usepackage{amsthm}
\usepackage{amsfonts}
\usepackage{tikz}
\usetikzlibrary{intersections}
\usepackage{bbm}
\usepackage{hyperref}
\usepackage{color}
\usepackage{enumitem}

\usepackage[explicit,indentafter]{titlesec}
\titleformat{\section}{\centering\normalfont\scshape}{\thesection.}{.5em}{#1}
\titleformat{\subsection}[runin]{\normalfont\itshape}{\textnormal{\thesubsection.}}{.5em}{#1.}
\titleformat{\subsubsection}[runin]{\normalfont\itshape}{\thesubsubsection.}{.5em}{#1.}
\titlespacing{\section}{0em}{1em}{0.5em}
\titlespacing{\subsection}{0em}{.5em}{0.5em}

\hypersetup{
colorlinks=true,
citecolor=[rgb]{0,0,0.5},
linkcolor=[rgb]{0,0,0.5},
urlcolor=[rgb]{0,0,0.75},
pdfpagemode=UseNone,
pdfstartview=FitH,
pdfdisplaydoctitle=true,
pdftitle={Spherical maximal functions and fractal dimensions of dilation sets},
pdfauthor={J. Roos, A. Seeger},
pdflang=en-US
}

\newcommand{\detail}[1]{}
\newcommand{\cmt}[1]{}

\theoremstyle{plain}
\newtheorem{thm}{Theorem}[section]

\newtheorem{lem}[thm]{Lemma}
\newtheorem{cor}[thm]{Corollary}
\newtheorem{prop}[thm]{Proposition}

\theoremstyle{remark}
\newtheorem{remark}[thm]{Remark}
\newtheorem*{remarka}{Remark}
\newtheorem*{defn}{Definition}

\def\BlackAndWhiteOnly{0}

\numberwithin{equation}{section}

\def\qA{\mathrm{qA}}
\def\W{{\mathcal W}}
\def\A{{\mathcal A}}
\def\R{\mathbb{R}}
\def\C{\mathbb{C}}
\def\Z{\mathbb{Z}}
\def\bbone{{\mathbbm 1}}
\def\a{\alpha}

\def\symb{\mathfrak{S}}
\def\lc{\lesssim}

\def\inn#1#2{\langle#1,#2\rangle}

\newcommand{\assouadconst}[1]{\chi_{\mathrm{A},\gamma}^{#1}(2^{-j})}
\newcommand{\jr}[1]{{\color{blue}{[JR: {#1}]}}}

\newcommand{\Qfourx}[1]{\d*(\d-1)/(\d*\d+2*#1-1)}
\newcommand{\Qfoury}[1]{(\d-1)/(\d*\d+2*#1-1)}
\newcommand\Qthreex[1]{(\d-#1)/( \d-#1+1)}
\newcommand\Qthreey[1]{1/(\d-#1+1)}

\newcommand{\definecoords}{
\def\ptsize{.1pt}
\def\QbgFillOpacity{.2}
\def\QbgFillColor{black}
\def\QbgLineOpacity{.6}
\def\QbgDrawCritSeg{1}
\def\QbgCritSegStyle{solid}
\def\QbgCritSegOpacity{\QbgLineOpacity}
\def\QbbFillOpacity{.1}
\def\QbbLineOpacity{.5}
\if\BlackAndWhiteOnly1
\def\AFillColor{black}
\def\AFillOpacity{.3}
\else
\def\AFillColor{red}
\def\AFillOpacity{.1}
\fi
\def\ALineOpacity{.8}
\coordinate (Q1) at (0,0);
\coordinate (Q2) at ( {(\d-1)/(\d-1+\b)}, {(\d-1)/(\d-1+\b)} );
\coordinate (Q3) at ( {\Qthreex{\b}}, { \Qthreey{\b} } );
\coordinate (Q30) at ( {\Qthreex{0}}, { \Qthreey{0} } );
\coordinate (Q4b) at ( { \Qfourx{\b} }, { \Qfoury{\b} } );
\coordinate (Q4g) at ( { \Qfourx{\g} }, { \Qfoury{\g} } );
\coordinate (R) at ( { \d*(\d-1)/ (\d*\d-1+\b) } , { (\d-1)/(\d*\d-1+\b) } );
\coordinate (C1) at ( { (\Qfourx{\b}+\Qfourx{\g})/2 }, {(\Qfoury{\b}+\Qfoury{\g})/2} );
\coordinate (C2) at (Q4b);
}

\newcommand{\drawauxlines}[2]{
\draw (0,0) [->] -- (0,1) node [left] {$\frac1q$};
\draw (0,0) [->] -- (1,0) node [below] {$\frac1p$};
\draw [dashed,opacity=.3] (1,0) -- (0,1);
\draw [dashed,opacity=.3] (#1) -- (1,{1/\d});
\draw [dashed,opacity=.3] (#2) -- (1,1);
\draw [dashed,opacity=.1]
(.5, 0) -- (.5, .5);
}

\def\theGamma{\gamma}
\newcommand{\drawQbg}{
\fill (Q1) node [left] {$Q_1$} circle [radius=.02em];
\fill (Q2) node [above left] {$Q_{2,\beta}$} circle [radius=\ptsize];
\fill (Q3) node [right] {$Q_{3,\beta}$} circle [radius=\ptsize];
\fill (Q4g) node [below right] {$Q_{4,\theGamma}$} circle [radius=\ptsize];
\fill [color=\QbgFillColor,opacity=\QbgFillOpacity] (Q1) -- (Q2) -- (Q3) -- (Q4g) -- cycle;
\draw [opacity=\QbgLineOpacity] (Q1) -- (Q2) -- (Q3);
\draw [opacity=\QbgLineOpacity] (Q4g) -- (Q1);
\if\QbgDrawCritSeg1
\draw [style=\QbgCritSegStyle,opacity=\QbgCritSegOpacity] (Q3) -- (Q4g);
\fi
}

\newcommand{\drawQbb}{
\fill [opacity=\QbbFillOpacity] (Q4g) -- (Q4b) -- (Q3) -- cycle;
\draw [opacity=\QbbLineOpacity] (Q4g) -- (Q4b) -- (Q3);
}

\newcommand{\lblQbb}{
\fill (Q4b) node [below right] {$Q_{4,\beta}$} circle [radius=\ptsize];
}

\newcommand{\lblQthreezero}{
\fill (Q30) node [below] {$Q_{3,0}$} circle [radius=\ptsize];
}

\newcommand{\lblrad}{
\fill (R) node [below right] {$R_\beta$} circle [radius=\ptsize];
}

\newcommand{\drawrad}{
\fill [opacity=.05] (Q4g) -- (R) -- (Q3) -- cycle;
\draw [opacity=.3] (Q4g) -- (R) -- (Q3);
}

\newcommand{\drawA}{
\fill [color=\AFillColor,opacity=\AFillOpacity] (Q4g) .. controls (C1) and (C2) .. (Q3);
\draw [opacity=\ALineOpacity] (Q4g) .. controls (C1) and (C2) .. (Q3);
}

\def\c{}

\newcommand{\cjr}{}
\newcommand{\cas}{}

\begin{document}
\title
[Spherical maximal functions]
{Spherical maximal functions and \\ fractal dimensions of dilation sets}
\author
[ J. Roos and A. Seeger]
{Joris Roos \ \ \ \ Andreas Seeger}

\thanks{Research supported in part by the National Science Foundation (DMS 1764295) and a Simons Fellowship.}


\address{Joris Roos \\ Department of Mathematics \\ University of Wisconsin \\480 Lincoln Drive\\ Madison, WI
53706, USA}

\curraddr{Department of Mathematics and Statistics, University of Massachusetts Lowell, USA}
\email{joris\_roos@uml.edu}
\address{Andreas Seeger \\ Department of Mathematics \\ University of Wisconsin \\480 Lincoln Drive\\ Madison, WI
53706, USA} \email{seeger@math.wisc.edu}

\begin{abstract}
For the spherical mean operators $\mathcal{A}_t$ in $\mathbb{R}^d$, $d\ge 2$, we consider the maximal functions $M_Ef =\sup_{t\in E} |\mathcal{A}_t f|$, with dilation sets $E\subset [1,2]$.
In this paper we give a surprising characterization of the closed convex sets which can occur as closure of the sharp $L^p$ improving region of $M_E$ for some $E$.
This region depends on the Minkowski dimension of $E$, but also other properties of the fractal geometry such as the Assouad spectrum of $E$ and subsets of $E$.
A key ingredient is an essentially sharp result on $M_E$ for a class of sets called (quasi-)Assouad regular which is new in two dimensions.
\end{abstract}

\maketitle

\section{Introduction}

For a locally integrable function $f$ on $\R^d$ with $d\ge 2$ let
\[\A_t f(x) = \int f(x-ty) d\sigma(y), \]
where $t>0$ and $\sigma$ denotes the normalized surface measure on the unit sphere in $\R^d$. Given a set $E\subset (0,\infty)$ consider the maximal function
\[ M_E f(x) = \sup_{t\in E} |\A_t f(x)|, \]
which is well--defined at least on continuous functions $f$. In this paper we study sharp {\em $L^p$ improving properties} of $M_E$. By scaling considerations it is natural to restrict attention to sets $E\subset [1,2]$. We define the {\em type set} $\mathcal{T}_E$ associated with $M_E$ by
\[ \mathcal{T}_E = \{ (\tfrac1p,\tfrac1q)\in [0,1]^2\,:\, M_E\;\;\text{is bounded}\;\; L^p\to L^q\}. \]
We are interested in determining for a given set $E$ the type set $\mathcal T_E$ up to the boundary, i.e. we will focus mainly on the closure of this set. Note that since $\mathcal{T}_E$ is by interpolation convex, the interior of $\mathcal{T}_E$ is determined by $\overline{\mathcal{T}_E}$.

We consider two natural problems. First, for each given $E\subset [1,2]$ the goal is to determine $\overline{\mathcal{T}_E}$.
Second, we ask which closed convex subsets of $[0,1]^2$ arise as $\overline{\mathcal{T}_E}$ for some $E\subset [1,2]$.

In this paper we give a complete solution to the second problem: we will determine exactly which closed convex sets arise as $\overline{\mathcal{T}_E}$. Moreover, we also give a satisfactory answer to the first problem for a large class of sets $E$ that covers all examples previously considered in the literature.

The case of a single average, $E=\{\text{point}\}$, is covered by a classical result of Littman \cite{littman}. Sharp results for the case $E=[1,2]$ are due to Schlag \cite{schlag}, Schlag and Sogge \cite{schlag-sogge} and S. Lee \cite{slee}. The case $p=q$ for the full spherical maximal operator goes back to Stein \cite{stein} in the case $d\ge 3$ and to Bourgain \cite{bourgain2} in the case $d=2$ (see also \cite{mss}). For some early results in special cases of dilation sets see \cite{cpcalderon}, \cite{cw} and \cite[p. 92]{Wai79}.
A satisfactory answer for general $E$ in the case $p=q$, depending on the Minkowski dimension of $E$, was given in \cite{sww1}; see also \cite{sww2}, \cite{stw} for refinements and related results. We remark that, while the question of sharp $L^p$ improving bounds is interesting in its own right, it is also motivated by problems on sparse domination and weighted estimates for global maximal functions $\sup_{t\in E}\sup_{k\in \mathbb Z} |\A_{2^kt} f|$, {\it cf.} \cite{BFP}, \cite{lacey}.

In a recent joint paper \cite{ahrs} with T. Anderson and K. Hughes we addressed the $L^p\to L^q$ problem for $M_E$ when $q>p$ in dimensions $d\ge 3$, with some partial results for $d=2$. It turned out that satisfactory results cannot just depend on the (upper) Minkowski dimension of $E$ alone and other notions of fractal dimension are needed, in particular the Assouad dimension.

Let us recall some definitions. Let $E\subset [1,2]$.
For $\delta>0$ let $N(E,\delta)$ the $\delta$-covering number, {\it i.e.} the minimal number of intervals of length $\delta$ required to cover $E$.
The {\em (upper) Minkowski dimension} $\dim_{\mathrm{M}}\!E$ of $E$ is
\[ \dim_{\mathrm{M}}\!E = \inf\big\{ a>0\,:\, \exists\,c>0\;\text{s.t.}\,\forall\,\delta\in(0,1),\, N(E,\delta)\le c\, \delta^{-a}\ \big\}. \]
The {\em Assouad dimension} $\dim_{\mathrm{A}}\!E$ (\cite{assouad}) is defined by
\begin{align*}
\dim_{\mathrm{A}}\!E = \inf\left\{ a>0\,:\, \right. \exists\,c>0\; & \text{s.t.}\,\forall\,I,\,\delta\in (0,|I|),\,\\
& \left. N(E\cap I,\delta)\le c\, \delta^{-a} |I|^a\ \right\};
\end{align*}
\cmt{\[
\inf\left\{ a>0\,:\, \right. \exists\,c>0\; \text{s.t.}\,\forall\,I,\,\delta\in (0,|I|),\,\left. N(E\cap I,\delta)\le c\, \delta^{-a} |I|^a\ \right\};
\]}
here $I$ runs over subintervals of $[1,2]$.
Note that $0\le \dim_{\mathrm{M}}\!E\le \dim_{\mathrm{A}}\!E\le 1$. For $0\le \beta\le \gamma\le 1$ let
\begin{equation}\label{eqn:Qdef}\begin{aligned}
& \quad\quad\quad Q_1 = (0,0),\;
Q_{2,\beta} = (\tfrac{d-1}{d-1+\beta}, \tfrac{d-1}{d-1+\beta}),\;\\
& Q_{3,\beta} = (\tfrac{d-\beta}{d-\beta+1}, \tfrac{1}{d-\beta+1}),\;
Q_{4,\gamma} = (\tfrac{d(d-1)}{d^2+2\gamma-1}, \tfrac{d-1}{d^2 + 2\gamma - 1}).
\end{aligned}\end{equation}
Moreover, let $\mathcal{Q}(\beta,\gamma)$ denote the closed convex hull of the points $Q_1$, $Q_{2,\beta}$, $Q_{3,\beta}$, $Q_{4,\gamma}$, see Figure \ref{fig:Qbg} below.
Let $\mathcal{R}(\beta,\gamma)$ denote the union of the interior of $\mathcal{Q}(\beta,\gamma)$ with the line segment connecting $Q_1$ and $Q_{2,\beta}$, including $Q_1$, but excluding $Q_{2,\beta}$.
The paper \cite{ahrs} gives a sufficient condition for $M_E$ to be $L^p\to L^q$ bounded, in dimension $d\ge 3$, namely if $\beta=\dim_{\mathrm M}\!E$, $\gamma_*=\dim_{\mathrm A}\!E$ then
\begin{equation}\label{eqn:rbg} \mathcal R(\beta,\gamma_*)\subset \mathcal T_E.\end{equation}
This inclusion was also obtained for $\gamma_*\le 1/2$ in two dimension, but the more difficult case $\gamma_*>1/2$ was left open. Our first main result is that \eqref{eqn:rbg} remains true for $d=2$, $\gamma_*>1/2$.

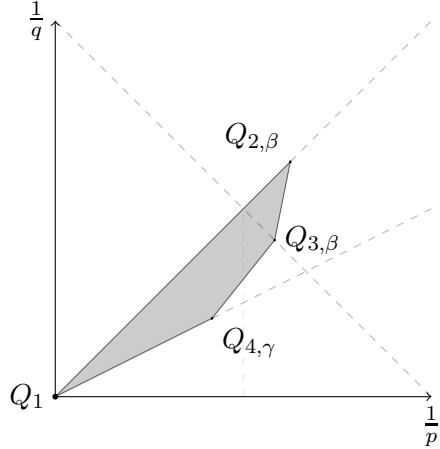
\begin{figure}[ht]
\begin{tikzpicture}[scale=5]
\def\d{2}
\def\b{.6}
\def\g{.9}
\definecoords
\drawauxlines{Q4g}{Q2}
\drawQbg
\end{tikzpicture}
\caption{The quadrangle $\mathcal{Q}(\beta,\gamma)$ for $d=2$, $\beta=0.6$, $\gamma=0.9$.}\label{fig:Qbg}
\end{figure}

We thereby get a rather satisfactory upper bound for $M_E$, which happens to be essentially sharp for so-called classes of Assouad regular sets discussed below. However, there is a slight shortcoming of this formulation which we will discuss now. Given $E$ the closure of the type set does not change if one replaces $E$ by its union with a set of zero Minkowski dimension; however
such unions may change the Assouad dimension (see \S
\ref{sec:minkzero}) and thus the set $\mathcal Q(\dim_{\mathrm{M}}\!E,\dim_{\mathrm{A}}\!E)$. To address this issue we replace
the notion of Assouad dimension with \emph{quasi-Assouad dimension}
introduced by L\"u and Xi in \cite{LX16} (see also \cite{fhhty}).

The definition involves certain intermediate fractal dimensions used in \cite{fhhty}, namely the
\emph{upper Assouad spectrum} $\theta \mapsto \overline\dim_{\mathrm{A},\theta}E$ which for given
$\theta \in [0,1]$ is defined by
\begin{align*}
\overline \dim_{\mathrm{A},\theta}E = \inf\left\{ a>0\,:\, \right. \exists\,c>0\; & \text{s.t.}\,\forall\,\delta\in (0,1),\,|I|\ge\delta^\theta,\,\\
& \left. N(E\cap I,\delta)\le c\, \delta^{-a} |I|^a\ \right\};
\end{align*}
here $I$ runs over subintervals of $[1,2]$.
The upper Assouad spectrum is a variant of the Assouad spectrum, where the condition $|I|\ge \delta^\theta$ is replaced by $|I|=\delta^\theta$. This was introduced by J. Fraser and H. Yu in \cite{fraser-yu1} (and used in \cite{ahrs} in the discussion of spherical maximal functions). The upper Assouad spectrum has the benefit that it is by definition nondecreasing in $\theta$.
One defines the quasi-Assouad dimension as the limit
\begin{equation}
\label{eqn:quasi-Assouad-def} \dim_{\qA}\!E = \lim_{\theta\to 1} \overline{\dim}_{\mathrm{A},\theta}E .\end{equation}
We remark that always $\dim_{\qA}\!E\le \dim_{\mathrm A}\!E $ and the inequality may be strict, see
\S\ref{sec:minkzero} for examples. With \eqref{eqn:quasi-Assouad-def} and $\mathcal R(\beta,\gamma)$ defined following \eqref{eqn:Qdef} we can now formulate
\begin{thm}\label{thm:upperbounds}
Let $d\ge 2$ and $E\subset [1,2]$, $\beta=\dim_{\mathrm{M}}\!E$, $\gamma=\dim_{\mathrm{qA}}\!E$. Then $\mathcal{R}(\beta,\gamma) \subset \mathcal{T}_E$.
\end{thm}
The most difficult case is $d=2$, $\gamma>\frac12$, and we will present the complete proof. In the cases $d\ge 3$ and $d=2$, $\gamma\le 1/2$ the result was essentially established in \cite{ahrs}, {\it cf.} \S\ref{sec:prelim} below for further review.

We shall now discuss the second problem mentioned above.
Modifications of well-known examples from \cite{schlag}, \cite{schlag-sogge}, \cite{sww1} (see \cite[\S 4]{ahrs} for details) show the lower bound
\begin{equation}\label{eqn:neccondition-beta}\overline{\mathcal{T}_E} \subset \mathcal{Q}(\beta,\beta)
\end{equation}
if $\beta=\dim_{\mathrm{M}}\!E$. Theorem \ref{thm:upperbounds} and \eqref{eqn:neccondition-beta} show that the set $\overline{\mathcal{T}_E}$ is a closed convex set satisfying the relation
$\mathcal{Q}(\beta, \gamma) \subset \overline{\mathcal{T}_E} \subset \mathcal{Q}(\beta, \beta)$
for $\gamma=\dim_{\qA}\!E.$
Surprisingly, this necessary condition on $\overline{\mathcal{T}_E}$ is also sufficient:
\begin{thm}\label{thm:type}
Let $\W\subset [0,1]^2$. Then

(i) $\W=\overline{\mathcal{T}_E}$ holds for some $E\subset [1,2]$ if and only if $\W$ is a closed convex set and \begin{equation}\label{eqn:Winclusion} \mathcal{Q}(\beta,\gamma)\subset \W\subset \mathcal{Q}(\beta,\beta)
\text{ for some $0\le \beta\le \gamma\le 1$. }
\end{equation}

(ii) For $\mathcal W=\overline{\mathcal T_E}$ in \eqref{eqn:Winclusion} one necessarily has $\dim_{\mathrm{M}}\!E=\beta$ and if in addition $\gamma$ is chosen minimally, then $\dim_{\qA}\!E=\gamma$.
\end{thm}
\begin{remark}\label{rem:assouad}
In the situation of (ii), for every $\gamma_*\in [\gamma, 1]$ the set $E$ can be chosen such that $\dim_{\mathrm A}\!E=\gamma_*$, {\it cf.} \S \ref{sec:type}.
\end{remark}

Figure \ref{fig:zoom} shows a more detailed look into the critical triangle spanned by the points $Q_{4,\gamma}$, $Q_{4,\beta}$, $Q_{3,\beta}$ and illustrates in particular that the boundary of $\mathcal{T}_E$ may follow an arbitrary convex curve in this triangle.

\begin{figure}[ht]
\begin{tikzpicture}
\def\d{2}
\def\b{.4}
\def\g{1}

\def\clipPad{.05}
\def\clipLx{ \Qfourx{\g} - \clipPad }
\def\clipLy{ \Qfoury{\g} -\clipPad }
\def\clipHx{ \Qthreex{\b} + \clipPad }
\def\clipHy{ \Qthreey{\b} + \clipPad }
\def\clipStyle{solid}
\def\clipOpacity{.1}

\def\clipCx{ (\clipLx+\clipHx)/2 }
\def\clipCy{ (\clipLy+\clipHy)/2 }
\def\cliprad{ .2 }

\def\theGamma{\gamma}

\begin{scope}[scale=4]
\definecoords

\coordinate (clipLL) at ({\clipLx}, {\clipLy});
\coordinate (clipHL) at ({\clipHx}, {\clipLy});
\coordinate (clipHH) at ({\clipHx}, {\clipHy});
\fill [opacity=.05] (clipLL) rectangle (clipHH);

\drawauxlines{Q4g}{Q2}
\drawQbg
\drawQbb
\cmt{
\drawrad
\lblrad}
\drawA
\end{scope}

\begin{scope}[xshift=1cm, yshift=-2cm, scale=15]
\definecoords

\coordinate (clipLLw) at ({\clipLx}, {\clipLy});
\coordinate (clipLHw) at ({\clipLx}, {\clipHy});
\coordinate (clipHHw) at ({\clipHx+.015}, {\clipHy});
\draw [\clipStyle, opacity=\clipOpacity] (clipHL) -- (clipLLw);
\draw [\clipStyle, opacity=\clipOpacity] (clipHH) -- (clipLHw);
\draw [\clipStyle, opacity=\clipOpacity] (clipLLw) rectangle (clipHHw);
\fill [opacity=.02] (clipLLw) rectangle (clipHHw);
\clip (clipLLw) rectangle (clipHHw);

\cmt{\drawauxlines{R}{Q2}}
\drawauxlines{Q4b}{Q2}
\drawQbg
\drawQbb
\lblQbb
\cmt{
\drawrad
\lblrad}
\drawA

\end{scope}

\end{tikzpicture}
\caption{}\label{fig:zoom}
\end{figure}

The basic idea of the proof of Theorem \ref{thm:type} is to write $\W$ as an at most countable intersection $\cap_{n} \mathcal Q(\beta_n,\gamma_n)$ and to construct the set $E$ in a suitable way as a disjoint union of sets $E_n$ with the property $\overline{\mathcal T_{E_n}}=\mathcal Q(\beta_n,\gamma_n)$. In order to implement this idea one needs to understand concrete cases in which Theorem \ref{thm:upperbounds} is sharp.

If $\dim_{\mathrm{qA}}\!E=\dim_{\mathrm{M}}\!E=\beta$, then by Theorem \ref{thm:upperbounds} and \eqref{eqn:neccondition-beta} we have $\overline{\mathcal{T}_E} = \mathcal{Q}(\beta,\beta)$. This happens for example if $E$ is a self-similar Cantor set of dimension $\beta$. In particular, Theorem \ref{thm:upperbounds} is sharp up to endpoints for such $E$.
The theorem is also sharp for a class of sets $E$ with $\dim_{\mathrm{M}}\!E<\dim_{\mathrm{qA}}\!E$.
Say that a set $E\subset[1,2]$ with $\dim_{\mathrm{M}}\!E=\beta$ and $\dim_{\mathrm{qA}}\!E=\gamma$ is {\em $(\beta,\gamma)$-regular} if either $\gamma=0$, or \[\overline{\dim}_{\mathrm{A},\theta} E=\dim_{\mathrm{\qA}}\!E\;\; \text{for all}\; 1>\theta>1-\beta/\gamma .\]
A set is {\em quasi-Assouad regular} if it is $(\beta,\gamma)$-regular for some $(\beta,\gamma)$. In \cite{ahrs} we have used a slightly more restrictive definition: a set is called {\em $(\beta,\gamma)$-Assouad regular} if the above condition holds with $\dim_{\qA}$ replaced by $\dim_A$ and {\em Assouad regular} if it is $(\beta,\gamma)$-Assouad regular for some $(\beta,\gamma)$. Assouad regular sets are also quasi-Assouad regular: for every Assouad regular set of positive Minkowski dimension we have $\dim_{\qA}\!E=\dim_{\mathrm A}\!E$.
Moreover, all sets with $\dim_{\mathrm{M}}\!E=0$ are quasi-Assouad regular since the condition is voidly satisfied when $\beta=0$. When $\beta=\dim_{\mathrm{M}}\!E=\dim_{\mathrm{\qA}}\!E$ the upper Assouad spectrum is constant, so $E$ is $(\beta,\beta)$-regular.

A convex sequence $E$ which has Minkowski dimension $\beta$ is $(\beta,1)$-regular. Other examples of (quasi-)Assouad regular sets can be found in \cite[\S 5]{ahrs}, see also \S\ref{sec:assouadcon} below for a refinement needed in the proof of Theorem \ref{thm:type}. The inclusion $\overline{\mathcal{T}}_E \subset \mathcal{Q}(\beta,\gamma)$ for $(\beta,\gamma)$-Assouad regular sets was proved in \cite[\S 4]{ahrs}. Here the maximal operator is tested on characteristic function of $\delta$-neighborhoods of spherical caps which have diameter
$\approx\sqrt{\delta^{\beta/\gamma}}$;
when $\beta=\gamma$ this reduces to a standard Knapp type example. We refer to \S\ref{sec:finiteunions} for a more general result. In this context we also note that for all $E$ the type set of $M_E$ when restricted to radial functions is strictly larger than the type set of $M_E$ on general functions ({\it cf.} \cite{rsradial}).

From the necessary conditions and Theorem \ref{thm:upperbounds} we have
\begin{equation}\label{eqn:typeset-ar}\overline{\mathcal{T}_E} = \mathcal{Q}(\beta,\gamma), \quad \text{ for $(\beta,\gamma)$-regular $E$,}\end{equation}
in all dimensions $d\ge 2$.
It turns out
that an essentially sharp result can be obtained for a much larger class, namely arbitrary finite unions of quasi-Assouad regular sets
in which case the closure of the type set is a closed convex polygon.
\begin{thm}\label{thm:finiteunions}
Let $d\ge 2$ and $E=\cup_{j=1}^m E_{j}$ where $E_{j}$ is $(\beta_j,\gamma_j)$-regular. Then $\overline{\mathcal{T}_E}=\cap_{j=1}^m \mathcal{Q}(\beta_j, \gamma_j)$.
\end{thm}
This is actually a simple consequence of Theorem \ref{thm:upperbounds} and the lower bounds, see \S \ref{sec:finiteunions}. Nevertheless, Theorem \ref{thm:finiteunions} is an essential step towards the proof of Theorem \ref{thm:type}. Moreover, Theorem \ref{thm:finiteunions} can be used to obtain certain sparse domination results on global spherical maximal functions, see \cite[\S6]{ahrs}.

It would be interesting to extend Theorem \ref{thm:finiteunions} to a wider class of sets.
Moreover, it is also worthwhile to investigate several endpoint results, {\it cf.} \S\ref{sec:endpoints} below.

\subsection*{Summary of the paper}
\begin{itemize}[leftmargin=.5cm]
\item[--]
In \S \ref{sec:prelim} we recall some previous results from \cite{sww1}, \cite{bourgain2}, \cite{ahrs} reducing the proof of Theorem \ref{thm:upperbounds} to Theorem \ref{thm:Q4} concerning the two-dimensional case with $\gamma\ge 1/2$. We also state a key ingredient, Corollary \ref{cor:interior}, for the proof of Theorem \ref{thm:type}. In \S\ref{sec:endpoints} we discuss some known and some open questions on endpoint estimates.

\item[--]
In \S \ref{sec:fracint} and \S \ref{sec:L2} we prove Theorem \ref{thm:Q4}. We use the general strategy from \cite{schlag-sogge}. Our main innovation here appears in \S \ref{sec:L2} and consists of the use of almost orthogonality arguments in conjunction with arguments based on the fractal geometry of the set $E$.

\item[--]
In \S \ref{sec:finiteunions}
we discuss a relevant necessary condition and prove Theorem \ref{thm:finiteunions}.

\item[--]
In \S \ref{sec:assouadcon} we present certain uniform constructions of (quasi-)Assouad regular sets. This is a refinement of \cite[\S 5]{ahrs}.

\item[--]
In \S \ref{sec:type} we prove Theorem \ref{thm:type}. This uses Theorem \ref{thm:upperbounds} (in the form of Corollary \ref{cor:interior}), \eqref{eqn:typeset-ar} and the construction in \S \ref{sec:assouadcon}.
\end{itemize}

\subsection*{Notation}
For a sublinear operator $T$ acting on functions on $\R^d$ we denote the $L^p\to L^q$ operator norm by $\|T\|_{p\to q}=\sup\{ \|T f\|_q\,:\,\|f\|_p=1 \}.$ Fourier transforms will be denoted by $\widehat{f}(\xi) = \int e^{-i\langle x,\xi\rangle} f(x) dx.$ Weighted $L^p$ spaces are denoted by $L^p(w)$ with $\|f\|_{L^p(w)} = \big( \int_{\R^d} |f(x)|^p w(x) dx \big)^{1/p}$. We will use $c$ to denote a positive constant that may change throughout the text and may depend on various quantities, which are either made explicit or clear from context. We write $A\lesssim B$ to denote existence of a constant $c$ such that $A\le c B$ and $A\approx B$ to denote $A\lesssim B$ and $B\lesssim A$.

\subsection*{Acknowledgement} We are grateful to the referee for a thorough reading of the paper and helpful suggestions.

\section{Setup and preliminary reductions}
\label{sec:prelim}
In this section we review and collect known facts about spherical averages from the literature. This will reduce the proof of Theorem \ref{thm:upperbounds} to the most difficult case, when $d=2$, $\gamma>\frac12$ and $(\frac1p,\frac1q)$ near $Q_{4,\gamma}$ (see Theorem \ref{thm:Q4} below).
At the same time, this review will pave the way for the proof of Theorem \ref{thm:type}, which requires a certain uniformity of various constants with respect to the set $E$.
Below we always assume $t\in [1,2]$.
\subsection{Dyadic decomposition}
Let $\chi$ be a smooth radial function on $\R^d$ supported in $\{1/2\le |\xi|\le 2\}$ such that $0\le \chi\le 1$ and $\sum_{j\in\Z} \chi(2^{-j}\xi)=1$ for every $\xi\not=0$. Set
\[ \chi_0(\xi) = 1-\sum_{j\ge 1} \chi(2^{-j} \xi),\;\text{and}\;\chi_j(\xi)=\chi(2^{-j} \xi)\quad \text{for}\;j\ge 1. \]
Next define $\A^j_t f$, $\sigma_{j,t}$ for $j\ge 0$ with
\begin{equation}\label{eqn:Ajdef}
\widehat{\A^j_t f}(\xi) = \chi_j(\xi) \widehat{\sigma}(t\xi) \widehat{f}(\xi)=\widehat {\sigma_{j,t}} (\xi)\widehat f(\xi).
\end{equation}
Then $A_t = \sum_{j\ge 0} \A^j_t$.

The symbol class $\symb^m$ is defined as the class of functions $a$ on $\R^d$ for which
\begin{equation}\label{eqn:stdsymb}
\|a\|_{\symb^m} = \max_{|\alpha| \le 10 d} \sup_{\xi\in \mathbb R^d}(1+|\xi|)^{-(m-|\alpha|)} |\partial^\alpha a (\xi)|
\end{equation}
is finite. Here $|\alpha|=\sum_{i=1}\alpha_i$ denotes the length of the multindex $\alpha\in \mathbb{N}^d_0$.
It is well--known (see \cite[Ch. VIII]{Steinbook3}) that
\begin{equation}\label{eqn:sigmaft}
\widehat{\sigma}(\xi) = \sum_{\pm} a_{\pm}(\xi) e^{\pm i |\xi|}
\end{equation}
with $a_{\pm}\in \symb^{-(d-1)/2}$.
Now \eqref{eqn:sigmaft} and Plancherel's theorem
imply
\begin{equation}\label{eqn:basicl2}
\| \A^j_t f\|_2 \lesssim 2^{-j \frac{d-1}2} \|f\|_2.
\end{equation}
The $L^1$ functions $\sigma_{j,t}$ satisfy the standard pointwise inequality
\[|\sigma_{j,t} (x)|\le C_N 2^j(1+2^j|\,|x|-t\,|)^{-N}\]
for $1\le t\le 2$ and thus we get
\begin{equation}\label{eqn:basicest1}
\|\A^j_t \|_{1\to 1}=\|\A^j_t\|_{\infty\to \infty}\lesssim 1
\end{equation}
and
\begin{equation}\label{eqn:basicest2}
\|\A^j_t \|_{1\to\infty} \lesssim 2^j .
\end{equation}
Appropriate interpolation among \eqref{eqn:basicl2}, \eqref{eqn:basicest1}, \eqref{eqn:basicest2} yields sharp estimates for the $L^p\to L^q$ operator norm of $A_t^j$
for each {\em fixed} $t$.

\subsection{Results near \texorpdfstring{$Q_1, Q_{2,\beta}, Q_{3,\beta}$}{Q1, Q2, Q3}}
We now turn our attention to the maximal operator associated with each $\A_t^j$. The uncertainty principle suggests that $|\A_t^j f(x)|$ is ``roughly constant''
as $t$ changes across an interval of length $\lesssim 2^{-j}$. Keeping in mind \eqref{eqn:basicl2}, \eqref{eqn:basicest1}, this suggests that for all $1\le p\le \infty$,
\begin{equation}\label{eqn:Mj-Lp}
\| \sup_{t\in E} |\A_t^j f| \|_p \le c N(E, 2^{-j})^{\frac1p} 2^{-j(d-1) \min(\frac1p,\frac{1}{p'})} \|f\|_p
\end{equation}
with $c$ only depending on $d$. This was proven in \cite{sww1} (also see \cite[Lemma 2.2]{ahrs}).
Observe that summing these estimates over $j\ge 0$ already gives $L^p\to L^p$ estimates for $M_E$ in the sharp range $p>1+\frac{\beta}{d-1}$ unless $d=2$ and $\beta=1$ (but this case is covered by Bourgain's circular maximal theorem \cite{bourgain2}).
In view of \eqref{eqn:basicest2}, the same argument (see \cite[Lemma 2.3]{ahrs}) also yields for $2\le q\le \infty$,
\begin{equation}\label{eqn:Mj-Lpq}
\| \sup_{t\in E} |\A_t^j f| \|_q \le c N(E, 2^{-j})^{\frac1q}
2^{-j(1-\frac{d+1}{q})} \|f\|_{q'}
\end{equation}
with $c$ only depending on $d$.
Appropriate interpolation of \eqref{eqn:Mj-Lp}, \eqref{eqn:Mj-Lpq} shows that $M_E$ is bounded $L^p\to L^q$ for every $(\frac1p,\frac1q)$ contained in the interior of the triangle with vertices $Q_1$, $Q_{2,\beta}$, $Q_{3,\beta}$ (see Figure \ref{fig:Qbg}).

\subsection{Minkowski and Assouad characteristics} \label{sec:char}
It is convenient to recast estimates involving $N(E,\delta)$ and $N(E\cap I, \delta)$ in terms of the following functions defined for $0<\delta<1$.

\begin{defn}
(i) The function
$\chi^E_{\mathrm{M},\beta}:(0,1]\to [0,\infty]$ defined by
\begin{equation}\label{eqn:mink-char}
\chi^E_{\mathrm{M},\beta}(\delta) = \delta^\beta N(E,\delta)
\end{equation}
is called the
{\it $\beta$-Minkowski characteristic} of $E$.

(ii) The function
$\chi^E_{\mathrm{A},\gamma}:(0,1]\to [0,\infty]$ defined by
\begin{equation}\label{eqn:assouad-char}
\chi^E_{\mathrm{A},\gamma}(\delta) = \sup_{\substack{ |I|\ge \delta}}\big(\tfrac{\delta}{|I|}\big)^\gamma N(E\cap I,\delta)
\end{equation}
is called the
{\it $\gamma$-Assouad characteristic} of $E$.
\end{defn}

The estimates \eqref{eqn:Mj-Lp}, \eqref{eqn:Mj-Lpq} can be rewritten as
\begin{equation}\label{eqn:Mjchar-Lp}
\| \sup_{t\in E} |\A_t^j f| \|_p \le c
[\chi^E_{\mathrm{M},\beta}(2^{-j})]^{\frac 1p} \times
\begin{cases}
2^{-j (\frac{d-1}{p'} -\frac \beta p)} \|f\|_p, &\text{ if } 1\le p\le 2,
\\
2^{-j (\frac{d-1-\beta}{p})} \|f\|_p, &\text{ if } 2\le p\le \infty,
\end{cases}
\end{equation}
\begin{equation}\label{eqn:Mjchar-Lpq}
\| \sup_{t\in E} |\A_t^j f| \|_q \le c [\chi^E_{\mathrm{M},\beta}(2^{-j})]^{\frac 1q}
2^{-j(1-\frac{d-\beta+1}{q})} \|f\|_{q'}, \quad 2\le q\le\infty.
\end{equation}

\subsection{Results near \texorpdfstring{$Q_{4,\gamma}$}{Q4}} This is the heart of the matter and here the Assouad characteristic enters.
\cmt{Let
\[ \assouadconst{E} =
\sup_{2^{-j} \le |I|\le 1} (2^{-j}/|I|)^{\gamma} N(E\cap I, 2^{-j} ). \]}
The cases $d\ge 3$ and $d=2$, $\gamma\le \frac12$ were already handled in \cite[\S 3]{ahrs}.
The analysis there is based on $TT^*$ arguments. Rewritten using \eqref{eqn:assouad-char}, it gives \begin{equation} \label{eqn:L2q}
\|\sup_{t\in E} |\A^j_t f|\|_{L^{q_\gamma,\infty}} \le c [\assouadconst {E}]^{\frac{1}{q_\gamma}} 2^{-j\frac{(d-1)^2-2\gamma}{2(d-1+2\gamma)} } \|f\|_{2},
\quad q_\gamma=\tfrac{2(d-1+2\gamma)}{d-1}.
\end{equation}
In the cases $d\ge 3$, and $d=2$, $\gamma<\frac12$ we have
$\frac{(d-1)^2-2\gamma}{2(d-1+2\gamma)} >0$
and by interpolation of
\eqref{eqn:L2q} with \eqref{eqn:basicest2} one obtains
\begin{equation}\label{eqn:ahrs}
\|\sup_{t\in E} |\A^j_t f|\|_{q_4} \le c\, [\assouadconst{E}]^{ 1/{q_4}} \|f\|_{p_4},\quad d\ge 3\;\text{or}\;d=2,\,\gamma< \tfrac12,
\end{equation}
where $c$ is a positive constants only depending on $d$ and
\[ Q_{4,\gamma}=(\tfrac1{p_4}, \tfrac1{q_4}) \]
as in \eqref{eqn:Qdef}. The remaining case is one of our main results in this paper.

\begin{thm}\label{thm:Q4} Let $d=2$ and $\gamma\ge 1/2$.
Then we have for every $j\ge 0$,
\begin{equation} \label{eqn:Q4main}
\big \|\sup_{t\in E} |\A^j_t f| \big\|_{q_4} \le c\, \min\big( \tfrac{j^{1/2}}{2\gamma-1}, j \big)^{1/p_4} {[\assouadconst{E}]}^{1/q_4} \|f\|_{p_4},
\end{equation}
where $c>0$ is an absolute constant. \cmt{check exp of $2\gamma-1$}
\end{thm}
The proof of this theorem is contained in \S \ref{sec:fracint} and \S \ref{sec:L2}. Interpolation arguments yield the following consequence, that implies Theorem
\ref{thm:upperbounds}.
\begin{cor}\label{cor:interior}
Let $d\ge 2$, $\beta=\dim_{\mathrm{M}}\!E$, $\gamma=\dim_{\mathrm{qA}}\! E$. Then
for every $(\frac1p,\frac1q)\in \mathcal Q(\beta,\gamma)$
there exists a nonnegative $\varepsilon=\varepsilon(\frac1p,\frac1q,\beta,\gamma,d)$ depending continuously on $(1/p, 1/q)$ such that $\varepsilon>0$ for $(\frac 1p, \frac 1q)$ in the interior of $\mathcal Q(\beta,\gamma)$ and on the open line segment between $Q_1$ and $Q_{2,\beta}$, and
\begin{equation}\label{eqn:Ajmain}
\| \sup_{t\in E} |\A_t^j f|\|_{q} \le
c\, [\assouadconst{E}]^{b_1} [\chi^E_{\mathrm{M},\beta}(2^{-j})]^{b_2}
(1+j)^{b_3}
2^{-\varepsilon j}
\|f\|_p.
\end{equation}
Here $c>0$ depends only on $d$, and the nonnegative constants $b_1, b_2, b_3$ satisfy $b_1+b_2=\frac1q$ and $b_3\le \tfrac2q$.
\end{cor}

\begin{proof} This follows by interpolation arguments using several extreme cases stated above (specifically, using \eqref{eqn:Mjchar-Lp}, \eqref{eqn:Mjchar-Lpq}, \eqref{eqn:ahrs}, \eqref{eqn:Q4main}). For the $L^\infty\to L^\infty$ estimate (corresponding to the pair $(0,0)=Q_1$) we have \eqref{eqn:Ajmain} with $b_1=b_2=b_3=0$ and $\varepsilon=0$.
For the pair $(p_2^{-1}, q_2^{-1})=Q_2(\beta)$ (here $p_2=q_2)$ we have \eqref{eqn:Ajmain} with
$b_1=0$, $b_2=1/q_2$, $b_3=0$ and $\varepsilon=0$.
For the pair $(p_3^{-1}, q_3^{-1})=Q_3(\beta)$ we have \eqref{eqn:Ajmain} with
$b_1=0$, $b_2=1/q_3$, $b_3=0$ and $\varepsilon=0$.
Finally we consider $(p_4^{-1}, q_4^{-1})=Q_4(\gamma)$. Now, in the case
$d=3$ or $d=2$, $\gamma<1/2$, we have
\eqref{eqn:Ajmain} with $b_1=1/q_4$, $b_2=0$, $ b_3=0$ and $\varepsilon=0$,
and in the case $d=2$, $\gamma\ge 1/2$ we have
\eqref{eqn:Ajmain} with $b_1=1/q_4$, $b_2=0$, $ b_3=2/q_4$ and $\varepsilon=0$.
We interpolate these estimates with the $L^2$ bound in
\eqref{eqn:Mjchar-Lp}, corresponding to $b_1=0$, $b_2=1/2$, $b_3=0$, $\epsilon= (d-1-\beta)/2$, except in the case $\beta=1$ and $d=2$ when we use Bourgain's result \cite{bourgain2} in the form of \cite{mss} for $p>2$.
\end{proof}

\subsection*{Proof of Theorem \ref{thm:upperbounds} }
We have $N(E\cap I,\delta) \le N(I,\delta)\le 2\delta^{-\varepsilon} $ if $\delta\le |I|\le \delta^{1-\varepsilon}$ and $0<\delta\le 1$.
Using the assumption on $\dim_{\qA}\!E$ we see that for any $\varepsilon>0$ and any $0<\varepsilon_1<1$ there are constants $C(\varepsilon, \varepsilon_1)<\infty$ such that for all intervals $I$ with $\delta\le |I|\le 1$
\[ N(E\cap I,\delta)\le \begin{cases}
C(\varepsilon, \varepsilon_1) (\delta/|I|)^{-\gamma-\varepsilon_1} &\text{ if } \delta^{1-\varepsilon} \le |I|\le 1
\\
2\delta^{-\varepsilon}
&\text{ if }\delta\le |I|\le \delta^{1-\varepsilon}.
\end{cases}
\] Thus for the $\gamma$-Assouad characteristic ({\it cf.} \S\ref{sec:char}) we get the estimate
\[\chi^E_{\mathrm{A},\gamma} (\delta) \le 2\delta^{-\varepsilon}+ C(\varepsilon,\varepsilon_1) \delta^{-\varepsilon_1},\quad 0<\delta\le 1,\] and we can conclude by applying Corollary \ref{cor:interior}. \qed

\begin{remarka} The interpolation argument above can be used to compute the exact exponent $\varepsilon$ in Corollary \ref{cor:interior} in terms of $\frac1p,\frac1q,\beta,\gamma,d$, but the exact dependence will not matter for us.
Moreover, the estimate \eqref{eqn:Q4main} is somewhat stronger than required: to prove the results stated in the introduction, it would suffice to show that for every $\epsilon_1>0$,
\[ \big \|\sup_{t\in E} |\A^j_t f| \big\|_{q_4} \lesssim_{\epsilon_1} 2^{j\epsilon_1} {[\assouadconst{E}]}^{1/q_4} \|f\|_{p_4}. \]
This weaker result together with interpolation arguments as above already implies Theorem \ref{thm:upperbounds}. The reason for stating \eqref{eqn:Ajmain} with the indicated degree of precision regarding dependence of the constant on the various parameters will become apparent in the proof of Theorem \ref{thm:type}, see \eqref{eqn:5_decay} below.
\end{remarka}

\subsection{Endpoint results and problems}\label{sec:endpoints} For the proof of Theorems \ref{thm:upperbounds} and \ref{thm:type} we do not have to consider endpoint questions. Nevertheless such endpoint bounds under the assumptions of bounded $\beta$-Minkowski characteristic or bounded $\gamma$-Assouad characteristic are very interesting, and some challenging problems are open.

We first consider the case $p=q$. Under the assumption $0<\beta<1$ it was noted in \cite[Prop. 1.4]{stw} that for the pair $Q_{2,\beta}$ the operator $M_E$ is of restricted weak type $(q_2, q_2)$, under the assumption of bounded $\beta$-Minkowski characteristic. This is proved by a version of Bourgain's interpolation argument in \cite{bourgain1}. It is conjectured (and suggested by the behavior on radial functions \cite{sww2}) that the restricted weak type estimate can be upgraded to a strong type $(q_2,q_2)$ estimate;
however this is known only for special types of sets $E$ such as convex sequences \cite{stw}, and is open for example for certain Cantor sets.

Now consider the case $p<q$. For the
full spherical maximal operator
Lee \cite{slee} proved a restricted weak type endpoint result for the exponent pairs $Q_{3,1}$ and $Q_{4,1}$, using the above mentioned Bourgain interpolation trick. This yields $L^p\to L^q$ estimates on the open edges $(Q_1,Q_{4,1})$ and $(Q_{3,1},Q_{4,1})$, moreover $L^{p,1}\to L^q$ bounds on the vertical half-open edge $(Q_{3,1},Q_{2,1}]$. The endpoint result for $Q_{4,1}$ is especially deep in two dimensions as it relies on Tao's difficult endpoint version \cite{tao-endpoint} of Wolff's bilinear adjoint restriction theorem for the cone \cite{wolff}.

The restricted weak type inequality at $Q_{3,\beta}$,
under the assumption of bounded $\beta$-Minkowski characteristic, was proved in \cite{ahrs}.
Under the assumption of bounded
$\gamma$-Assouad characteristic,
if $d\ge 3$ or $d=2$, $\gamma<1/2$ the restricted weak type estimates for $Q_{4,\gamma}$ was also proved in \cite{ahrs}. We remark that for these known restricted weak type endpoint estimates at $Q_{3,\beta}$ and $Q_{4,\gamma}$ it is open whether they can be upgraded to strong type estimates.

Endpoint bounds at $Q_{4,\gamma}$ with bounded Assouad characteristic are open in two dimensions when $1/2\le \gamma<1$.
We conjecture that the term $(1+j)^{b_3}$ in Corollary \ref{cor:interior} can be dropped; moreover that a restricted weak type estimate holds at $Q_{4,\gamma}$.

\section{Proof of Theorem \ref{thm:Q4}: Fractional integration}\label{sec:fracint}

In this section we prove Theorem \ref{thm:Q4}. For $j=0$ there is nothing to prove, so we assume $j\ge 1$ from here on. For $a\in \symb^0, t\in [1,2]$ define
\begin{equation}\label{eqn:Tjdef}
T_t^{j,\pm}[a,f] (x) = \int_{\R^2} e^{i \langle x,\xi\rangle \pm i t |\xi|} \chi_j(\xi) a(t\xi) \widehat{f}(\xi)d\xi\quad(x\in\R^2).
\end{equation}
From \eqref{eqn:Ajdef} and \eqref{eqn:sigmaft} we see that there exist symbols $a_{\pm}$ with $\|a_{\pm} \|_{\symb^{0}} \le C(d)$ such that
\[ \A^j_t f = 2^{-j/2}\sum_{\pm} T_t^{j,\pm}[a_{\pm}, f]. \]
In the following let us assume without loss of generality that
\begin{equation} \label{eqn:a-normalized}
\|a\|_{\symb^0}\le 1\quad\text{and}\quad \text{supp}\,a\subset \{ \xi\,:\, 0<\xi_1<2^{-10}\xi_2 \}.
\end{equation}
A first observation is that the effect from oscillation of the factor $e^{\pm i t|\xi|}$ in \eqref{eqn:Tjdef} is negligible if $t$ varies within an interval of length $\ll 2^{-j}$. This suggests the following standard argument.
Define
\[ I_{n,j} = [n2^{-j}, (n+1) 2^{-j}] \quad (n\in\Z),\]
\begin{equation}\label{eqn:Ejdef}
\mathcal{E}_j = \{ n 2^{-j} \,:\, I_{n,j}\cap E\not=\emptyset,\,n\in \Z\}\subset [1,2].
\end{equation}
We estimate pointwise for each $x\in\R^2$,
\[ \sup_{t\in E} |T_t^{j,\pm}[a,f]|(x) \le \Big(\sum_{n\in 2^j\mathcal{E}_j} [\sup_{t\in I_{n,j}} |T_t^{j,\pm}[a,f]|(x)]^{q_4} \Big)^{1/q_4}. \]
For every $n\in 2^j \mathcal{E}_j$ and $t\in I_{n,j}$ we use the fundamental theorem of calculus to estimate
\[ |T_t^{j,\pm}[a,f](x)| \le |T_{n2^{-j}}^{j,\pm}[a,f](x)| + \int_{0}^{2^{-j}} |T_{n2^{-j}+s}^{j,\pm}[\widetilde{a},f](x)| ds, \]
where $\widetilde{a}\in \symb^{1}$, more precisely $\widetilde{a}(\xi) = \pm i |\xi| a(\xi) + \langle \xi,\nabla a(\xi)\rangle$, and we have used that $t\ge 1$.
From the previous two displays,
\begin{equation}\label{eqn:fundthmcalc}
\|\sup_{t\in E} |T_t^{j,\pm} [a,f]| \|_{q} \le \Big( \sum_{t\in \mathcal{E}_j} \|T_t^{j,\pm} [a,f]\|_{q}^{q} \Big)^{\frac1q} + \int_0^{2^{-j}} \Big( \sum_{t\in \mathcal{E}_j+s} \|T_{t}^{j,\pm} [\widetilde{a},f]\|_q^q \Big)^{\frac1q} ds,
\end{equation}
where $q=q_4=3+2\gamma\ge 4$.
The first term on the right hand side and the integrand of the second term will be treated in the same way. Bilinearizing, we write
\begin{equation}\label{eqn:bilinearized} \Big(\sum_{t\in\mathcal{E}}\|T^{j,\pm}_t [a,f]\|^q_q\Big)^{1/q} = \| \mathcal{T}_j (f\otimes f)\|_{L^{q/2}(\R^2\times\mathcal{E})}^{1/2},
\end{equation}
where $\mathcal{E}\subset [1,2]$ is a finite set, $\R^2\times\mathcal{E}$ is equipped with the product of the Lebesgue measure and the counting measure, $(f\otimes f)(x,y) = f(x)f(y)$ and
\[ \mathcal{T}_j F(x,t) = \int_{\R^4} e^{i\langle x,\xi+\zeta\rangle\pm i t (|\xi|+|\zeta|)} \chi_j(\xi) a(t\xi) \chi_j(\zeta) a(t\zeta) \widehat{F}(\xi,\zeta)\,d(\xi,\zeta). \]

\begin{defn}
Let $\delta\in (0,1)$. A finite set $\mathcal{E}\subset (0,\infty)$ will be called {\em uniformly $\delta$--separated} if
\begin{equation}\label{eqn:unifsep}
\mathcal{E}-\mathcal{E} \subset \delta \Z.
\end{equation}
In other words, $\mathcal{E}$ is a subset of an arithmetic progression with common difference of $\delta$.
\end{defn}

The sets $\mathcal{E}_j$ and $\mathcal{E}_j+s$ appearing in \eqref{eqn:fundthmcalc} are uniformly $2^{-j}$--separated. We have a crucial $L^2$ estimate with the following weight that blows up near the diagonal:
\begin{equation}\label{eqn:wgdef}
w_\gamma(y,z) = |y-z|^{-(2\gamma-1)}\quad\text{for}\; y,z \in \R^2,\, y\not=z.
\end{equation}

\begin{prop}\label{prop:l2} Let $\gamma\ge \frac12$ and let $a\in\symb^0$ satisfy \eqref{eqn:a-normalized}.
Assume that $j\ge 1$ and that $\mathcal{E}\subset [1,2]$ is uniformly $2^{-j}$--separated. Then
\begin{equation}\label{eqn:l2-weighted}
\|\mathcal{T}_j F\|_{L^2(\R^2\times \mathcal{E})} \le c
\,\min\big( \tfrac{j^{1/2}}{2\gamma-1}, j\big) \, 2^{j} [\assouadconst{\mathcal{E}}]^{\frac 1{2}} \|F\|_{L^2(w_\gamma)}\,,
\end{equation}
where
$c$ is an absolute constant. \cas
\end{prop}
The proof of this estimate forms the heart of the matter and is contained in Section \ref{sec:L2}.

\begin{proof}[Proof of Theorem \ref{thm:Q4} given Proposition \ref{prop:l2}] We assume that $\|a\|_{\symb^0}\le 1$.
Since $q_4=2p_4=3+2\gamma$, we need to show, by \eqref{eqn:fundthmcalc} and \eqref{eqn:bilinearized}, that
\begin{equation}\label{eqn:lq-weighted}
\|\mathcal{T}_j (f\otimes f)\|_{L^{p_4}(\R^2\times \mathcal{E})} \le c
\,\min\big( \tfrac{j^{1/2}}{2\gamma-1}, j\big)^{2/p_4} \, 2^{j} [\assouadconst{\mathcal{E}}]^{1/p_4} \|f\|_{p_4}^2\,.
\end{equation}
The case $\gamma=1/2$ follows immediately from Proposition \ref{prop:l2}.
We assume $1/2<\gamma\le 1$ and argue as in the paper by Schlag and Sogge \cite{schlag-sogge}.
The sectorial localization in \eqref{eqn:a-normalized} allows us to estimate \cas \cjr
\begin{equation} \label{L1Linfty}
\sup_{(x,t)\in \R^2\times [1,2]} |\mathcal{T}_j F(x,t)| \lesssim 2^j
\int_{\R^2} \sup_{(y_2,z_2)\in\R^2} |F(y,z)| d(y_1,z_1).
\end{equation}
Indeed, we have the pointwise bound, for $1\le t\le 2$,
\begin{equation}\label{eqn:Linftyest} | \mathcal{T}_j F(x,t)| \lesssim\int_{\R^4} K_{j,t} (x-y, x-z) |F(y,z)|\, d(y,z)
\end{equation}
\[
\le \Big(\sup_{(y_1,z_1)\in\R^2} \int_{\R^2} |K_{j,t} (y,z)| d(y_2, z_2) \Big) \cdot \Big(\int_{\R^2} \sup_{(y_2,z_2)\in\R^2} |F(y,z)| d(y_1,z_1)\Big)
\]
where, with $h(s)=\sqrt{1-s^2}$, $K_{j,t}(y,z)$ is a linear combination of five terms
\begin{multline*}
\frac{2^j }{(1+t^{-1}(|y|+|z|))^{10}},\quad 2^j \bbone_{[-1/2,1/2]^2}(y_1,z_1) \times \\
\frac{ 2^j}{(1+2^j|t^{-1}y_2\pm h(t^{-1}y_1)|)^{10}}\frac{ 2^j}{(1+2^j|t^{-1}z_2\pm h(t^{-1} z_1)|)^{10}}.
\end{multline*} In \eqref{eqn:Linftyest} we integrate in $(y_2,z_2)$ first to obtain \eqref{L1Linfty}.
In \eqref{eqn:l2-weighted} we may replace $ |y-z|^{-\frac{2\gamma-1}{2}} $ with
$|y_1-z_1|^{-\frac{2\gamma-1}{2}} $ and then analytically interpolate the resulting inequality with \eqref{L1Linfty}. This yields for $2\le p< \infty$,
\begin{multline}
\|\mathcal{T}_j F\|_{L^p(\R^2\times \mathcal{E}_j)}
\lesssim 2^j
\min\big( \tfrac{j^{1/2}}{2\gamma-1}, j\big)^{2/p}
[\assouadconst{E}]^{1/p}\times\\
\Big(\int_{\R^2} \Big[ |y_1-z_1|^{(1-2\gamma)/{p} } \Big(\int_{\R^2} |F(y,z) |^{p} d(y_2, z_2)\Big)^{1/p} \Big]^{p'} d(y_1, z_1)\Big)^{1/p'}.
\end{multline}
We specialize to $F(y,z)= f(y) g(z)$ and apply H\"older's inequality in $y_1$ with exponents $p/p'\in [1,\infty]$ and $(p/p')'$ to obtain
\begin{align}
&\Big(\int_{\R^2} \Big[ |y_1-z_1|^{(1-2\gamma)/{p} } \Big(\int_{\R^2} |f(y)g(z)|^{p} d(y_2, z_2)\Big)^{1/p} \Big]^{p'} d(y_1, z_1)\Big)^{1/p'}
\notag
\\&=
\Big(\int_{\R^2} \big[ |y_1-z_1|^{\frac {1-2\gamma}{p} } \|f(y_1,\cdot)\|_{p}
\|g(z_1,\cdot)\|_{p} \big]^{p'}\,d(y_1, z_1)\Big)^{1/p'}
\notag
\\
\label{fractint}
&\le \|f\|_p \Big(\int_\R\Big[\int_\R |y_1-z_1|^{- (2\gamma-1)\frac{p'}{p}}
\|g(z_1,\cdot)\|_p^{p'} dz_1\Big] ^{(\frac p{p'})' }dy_1\Big)^{\frac{1}{(p/p')'}\frac{1}{p'}}.
\end{align}
The standard fractional integral theorem says that for $0<\mathrm{Re}(a)\le 1$ the convolution operator with Schwartz kernel
$|s-t|^{a-1}$ maps $L^{\mathfrak{p}}(\R)$ to $L^\mathfrak{q}(\R)$ for $\mathfrak{p}^{-1}-\mathfrak{q}^{-1}=\mathrm{Re}(a)$, and using analytic interpolation with the trivial $L^1\to L^\infty$ estimate when $\mathrm{Re}(a)=1$ one notes that the operator norm is bounded as $\mathrm{Re}(a)\to 1$.
For $p>2$ and $1/2<\gamma\le 1$ the expression
$1-(2\gamma-1) {p'}/{p}$ belongs to $(0,1-p'/p)$. Thus, if $p>2$ and $1< r< (p/p')'$ is defined by
\begin{equation}\label{rdefinition} \tfrac 1r-\big(1-\tfrac{p'}{p}\big)= 1-(2\gamma-1) \tfrac{p'}{p}\,=: a
\end{equation} then we
see that \eqref{fractint} is bounded by
\[ C\|f\|_p \Big(\int_\R\|g(z_1,\cdot)\|_{p}^{p' r} dz_1\Big)^{\frac 1{rp'} } .
\] where the constant is independent of $\gamma\in (1/2,1]$.
For the special case $r=p/p'$ the relation \eqref{rdefinition} gives $p=(3+2\gamma)/2=p_4$ (which is $>2$ since $\gamma>1/2$), and we get \eqref{eqn:lq-weighted} by setting $f=g$.
\end{proof}

\section{Proof of Proposition \ref{prop:l2}: An \texorpdfstring{$L^2$}{L2} estimate}\label{sec:L2}
In this section we prove the crucial $L^2$ estimate, Proposition \ref{prop:l2}.
As in Schlag--Sogge \cite{schlag-sogge}, the key to this estimate will be a second dyadic decomposition in the angle $\measuredangle(\xi,\eta)\in [0,\pi]$ between certain frequency variables $\xi$ and $\eta$.
For the estimates in \cite{schlag-sogge} the authors relied on space time estimates due to Klainerman and Machedon \cite{klainerman-machedon} which are not applicable in our setting. Instead we have to establish an orthogonality property between contributions from different values of $t$ which also depends on the fractal geometry of $E$.

We find it convenient to introduce the notation
\begin{equation}\label{eqn:def-S}
S[F,b](\xi,t)= \int_{\R^2} e^{\pm i t(|\xi-\eta|+|\eta|)} b(t,\xi,\eta)
\widehat F(\xi-\eta,\eta) d\eta,
\end{equation}
acting on a function $F:\R^2\times\R^2\to\C$ and a symbol $b:(0,\infty)\times \R^2\times \R^2\to \C$.
For positive integers $j,m$ we set $a_j(t,\xi)=a(t\xi)\chi_j(\xi)$ and
\[ b_{j}(t,\xi,\eta) = a_j(t,\xi-\eta)a_j(t,\eta), \]
\begin{equation}\label{eqn:l2-defbjm}
b_{j,m}(t,\xi,\eta) = b_{j}(t,\xi,\eta) \chi(2^{m-2j} \det(\xi,\eta)).
\end{equation}
Define the convex annular sector
\[\Theta_{j} = \{\omega\in\R^2:0<\omega_1< 2^{-10}\omega_2,\, 2^{j-1}\le |\omega|\le 2^{j+1}\}.\]
If $b_j(t,\xi,\eta)\not=0$, then both $\eta$ and $\xi-\eta$ lie in $\Theta_{j}$ (see \eqref{eqn:a-normalized}).
By convexity, we also have $\frac12\xi= \frac12\eta + \frac12(\xi-\eta)\in \Theta_{j}$.
Observe that the cutoff in \eqref{eqn:l2-defbjm} effects an angular localization, since
\begin{equation}\label{eqn:l2-detangle}
|\det(\xi,\eta)| = |\xi|\cdot |\eta|\cdot \sin(\measuredangle(\xi,\eta)).
\end{equation}
Note that
\[ \mathcal{T}_j F(x,t) = \int_{\R^2} S[F,b_j](\xi, t) e^{i\langle \xi,x\rangle} d\xi. \]
Hence by Plancherel's theorem, Proposition \ref{prop:l2} follows if we can show
\begin{equation}\label{eqn:l2-penult}
\|S[F,b_j]\|_{L^2(\R^2\times\mathcal{E})} \le c
\,\min\big( \tfrac{j^{1/2}}{2\gamma-1}, j\big) [\assouadconst{\mathcal{E}}]^{\frac 1{2}} 2^j \|F\|_{L^2(w_\gamma)}.
\end{equation}
We will require three different estimates for the objects $S[F,b]$ to prove this estimate. To clarify various dependencies we introduce the following terminology.

\begin{defn}
Let $j,m\ge 1$.
We say that the symbol $b$ is {\it $(j,m)$-adapted} if $b$ is smooth and
$b(t,\xi,\eta)=0$ unless
\begin{equation}\label{eqn:l2-defadapted}
\eta, \xi-\eta \in \Theta_{j}\quad\text{and}\quad \measuredangle(\xi, \eta) \le 2^{-m+5}.
\end{equation}
We call $b$ \emph{strictly $(j,m)$-adapted} if in addition $b(t,\xi,\eta)=0$ unless
\[\measuredangle(\xi,\eta)\ge 2^{-m-5}.\]
\end{defn}

Observe that $b_{j,m}$ is strictly $(j,m)$-adapted.
\begin{prop}[Trivial estimate]\label{prop:l2-triv}
If $b$ is $(j,m)$-adapted, then for every finite $\mathcal{E}\subset [1,2]$,
\begin{equation}\label{eqn:l2-triv}
\| S[F,b] \|_{L^2(\R^2\times\mathcal{E})} \le c\, \|b\|_\infty 2^{j-\frac{m}2} (\# \mathcal{E})^{\frac12} \|F\|_{L^2(\R^4)}.
\end{equation}
\end{prop}
\begin{proof}
This follows by an application of the Cauchy--Schwarz inequality to the integration over $\eta$ in \eqref{eqn:def-S}.
\end{proof}
The next ingredient is the following crucial improvement of \eqref{eqn:l2-triv}.
\begin{prop}[Almost orthogonality]
\label{prop:l2-orth}
Suppose that $b$ is strictly $(j,m)$-adapted and satisfies the differential inequality
\begin{equation}\label{eqn:l2-angular-symb}
| \langle \tfrac{\eta}{|\eta|}, \nabla\rangle^N b(t,\xi,\eta)| \le B\, 2^{-jN}\;\text{for}\;N=0,1,2.
\end{equation}
Suppse $\mathcal{E}$ is uniformly $2^{-j}$--separated. Then
\begin{equation}\label{eqn:l2-orth}
\| S[F,b] \|_{L^2(\R^2\times\mathcal{E})} \le c\, B\, 2^{j-\frac{m}2} \big(\sup_{|I|=2^{-j+2m}} \#(\mathcal{E}\cap I) \big)^{\frac12} \|F\|_{L^2(\R^4)}.
\end{equation}
\end{prop}
The proof of this proposition is contained in \S \ref{sec:mainorth}. It relies on an observation of almost orthogonality of $S[F,b](\cdot,t)$ and $S[F,b](\cdot,t')$ for sufficiently separated $t,t'$.
Finally we will need a certain tail estimate.
\begin{prop}[Off--diagonal decay]
\label{prop:l2-offdiag}
Let $0\le m\le j/2$. Suppose that $b$ is $(j,m)$-adapted and satisfies the differential inequality
\begin{equation}\label{eqn:l2-bad-symb}
|\partial_\eta^\alpha b(t,\xi,\eta)| \le B\, 2^{-(j-m)|\alpha|}\;\text{for}\;|\alpha|\le 9.
\end{equation}
Assume that $F$ is supported on the set
\[ \{ (y,z)\,:\,|y-z|\ge 2^{-m+\ell+20} \} \]
for some $\ell\ge 0$. Then for every finite $\mathcal{E}\subset [1,2]$, \begin{equation}\label{eqn:l2-offdiag}
\| S[F,b] \|_{L^2(\R^2\times\mathcal{E})} \le c\, B\, 2^{\frac32 m-\ell} (\mathcal{\# E})^{\frac12}\|F\|_{L^2(\R^4)}.
\end{equation}
\end{prop}
The proof of this proposition is contained in \S \ref{sec:offdiag}.
It may be helpful to recognize \eqref{eqn:l2-offdiag} as the special case $N=1$ of the stronger estimate
\[ \| S[F,b] \|_{L^2(\R^2\times\mathcal{E})} \le c_N\, B\, 2^{j-\frac m2} 2^{-(j-2m+\ell)N} (\mathcal{\# E})^{\frac12}\|F\|_{L^2(\R^4)}, \]
which also holds (as long as \eqref{eqn:l2-bad-symb} holds for large enough $|\alpha|$), but will not be needed for our purpose.

Note that \eqref{eqn:l2-bad-symb} features derivatives taken in arbitrary directions, whereas the derivatives in \eqref{eqn:l2-angular-symb} are taken in the radial direction only. The difference between these two estimates reflects the fact that for fixed $\xi$, the $\eta$--support of $b$ is contained in a rectangle of dimensions, say $2^{j+10}\times 2^{j-m+10}$, with its long side aligned radially.

\subsection{Proof of Proposition \ref{prop:l2} given Propositions \ref{prop:l2-orth} and \ref{prop:l2-offdiag}} It suffices to show \eqref{eqn:l2-penult}.
We first show the uniform estimate
\begin{equation}\label{eqn:l2-weighted-unif}
\|S[ F, b_j]\|_{L^2(\R^2\times \mathcal{E})} \le c\,j 2^{j} [\assouadconst{\mathcal{E}}]^{\frac 1{2}} \|F\|_{L^2(w_\gamma)}.
\end{equation}

Observe that since $\mathcal{E}$ is uniformly $2^{-j}$--separated,
\begin{equation}\label{eqn:l2-ejgj}
\#\mathcal{E} \lesssim 2^{j\gamma} \assouadconst{\mathcal{E}}\quad\text{and}\quad \sup_{|I|=2^{-j+2m}} \#(\mathcal{E}\cap I)\lesssim 2^{2m\gamma} \assouadconst{\mathcal{E}}.
\end{equation}
Thus, the almost orthogonality estimate \eqref{eqn:l2-orth} only beats the trivial estimate \eqref{eqn:l2-triv} if $m<j/2$. This motivates the definition of the remainder term
\[ R_j F = S\Big[F, b_j - \sum_{0<m<j/2} b_{j,m} \Big] \]
so that $S[F,b_j] = \sum_{0<m<j/2} S[F,b_{j,m}] + R_j F$.
Observe that the symbol of $R_j$ is $(j,j/2)$--adapted.
Estimate
\begin{multline}\label{eqn:basic-S-dec-simple} \|S[ F, b_j] \|_{L^2(\R^2\times\mathcal{E})} \le \sum_{0<m<j/2} \|S[F,b_{j,m}]\|_{L^2(\R^2\times\mathcal{E})} + \|R_j F\|_{L^2(\R^2\times\mathcal{E})}.
\end{multline}
We estimate both terms separately.
To do this we write $F$ as
\[ F = \sum_{k\in\Z} F_k = \sum_{k\le -m} F_k + \sum_{\ell=1}^\infty F_{-m+\ell}, \]
with each $F_k$ supported on the set
\[ \{(y,z)\in\R^2\times\R^2 \,:\, 2^{k}\le 2^{-20} |y-z|< 2^{k+1} \}. \]
Note that $b_{j,m}$ is strictly $(j,m)$-adapted and satisfies both, \eqref{eqn:l2-angular-symb} and \eqref{eqn:l2-bad-symb}.
Proposition \ref{prop:l2-orth} and \eqref{eqn:l2-ejgj} imply
\[ \| S\big[\textstyle\sum_{\tiny k\le -m} F_k,b_{j,m}\big] \|_{L^2(\R^2\times\mathcal{E}_j)}\lesssim 2^{j} [\assouadconst{\mathcal{E}}]^{\frac12} \|F\|_{L^2(w_\gamma)}, \]
where we have used that $|y-z|^{-(1-2\gamma)/2}\lesssim 2^{-m(\gamma-\frac12)}$ on the support of $\textstyle\sum_{k\le -m} F_k$. On the other hand, for $\ell\ge 1$, Proposition \ref{prop:l2-offdiag} and \eqref{eqn:l2-ejgj} yield
\[ \|S[F_{-m+\ell}, b_{j,m}]\|_{L^2(\R^2\times \mathcal{E})} \lesssim 2^{-\ell+\frac32 m + \frac{\gamma}2 j} [\assouadconst{\mathcal{E}}]^{\frac12} \|F_{-m+\ell}\|_{L^2(\R^4)}. \]
Since $|y-z|^{-(1-2\gamma)/2}\approx 2^{(\ell-m)(\gamma-\frac12)}$ on the support of $F_{-m+\ell}$ the quantity on the right hand side is comparable to
\begin{align*} 2^{-\ell(\frac32-\gamma)+m(2-\gamma)+j \frac{\gamma}2} [\assouadconst{\mathcal{E}}]^{\frac12} \|F_{-m+\ell}\|_{L^2(w_\gamma)} &\\ \lesssim 2^{-\frac{\ell}2+j} [\assouadconst{\mathcal{E}}]^{\frac12} \|F\|_{L^2(w_\gamma)},& \end{align*}
where we used that $m\le j/2$ and $\gamma\le 1$ in the last step.
Together we obtain
\[ \sum_{0<m<j/2} \|S[F,b_{j,m}]\|_{L^2(\R^2\times\mathcal{E})} \lesssim j 2^j [\assouadconst{\mathcal{E}}]^{\frac12} \|F\|_{L^2(w_\gamma)}. \]
The estimate for the second term in \eqref{eqn:basic-S-dec-simple} is similar:
Proposition \ref{prop:l2-triv} implies
\[ \|R_j \big[ \sum_{k\le -m} F_k \big] \|_{L^2(\R^2\times\mathcal{E})} \lesssim 2^j [\assouadconst{\mathcal{E}}]^{\frac12} \|F\|_{L^2(w_\gamma)}. \]
On the other hand, applying Proposition \ref{prop:l2-offdiag} with $m=j/2$ gives
\[ \| R_j [F_{\lceil-j/2\rceil+\ell} ] \|_{L^2(\R^2\times\mathcal{E})} \lesssim 2^{-\frac\ell{2} + j} [\assouadconst{\mathcal{E}}]^{\frac12} \|F\|_{L^2(w_\gamma)} \]
for all $\ell \ge 1$. The previous two displays combined give
\begin{equation}\label{eqn:Rjest} \|R_j F \|_{L^2(\R^2\times\mathcal{E})} \lesssim 2^j [\assouadconst{\mathcal{E}}]^{\frac12} \|F\|_{L^2(w_\gamma)}, \end{equation}
as required.

In order to finish the proof we need to establish
an improvement over \eqref{eqn:l2-weighted-unif} in the range $j>(2\gamma-1)^{-2}$, namely
\begin{equation}\label{eqn:l2-weighted-impr}
\|S[ F,b_j]\|_{L^2(\R^2\times \mathcal{E})} \le c\,(2\gamma-1)^{-1}j^{1/2} 2^{j} [\assouadconst{\mathcal{E}}]^{\frac 1{2}} \|F\|_{L^2(w_\gamma)}.
\end{equation}
Estimate
\begin{multline}\label{eqn:basic-S-dec} \|S[ F, b_j] \|_{L^2(\R^2\times\mathcal{E})} \\ \le j^{\frac12} \Big( \sum_{0<m<j/2} \|S[F,b_{j,m}]\|_{L^2(\R^2\times\mathcal{E})}^2 \Big)^{1/2} + \|R_j F\|_{L^2(\R^2\times\mathcal{E})}.
\end{multline}
The second term has already been estimated in \eqref{eqn:Rjest}.
To treat the first term we first observe that
by Proposition \ref{prop:l2-offdiag} and \eqref{eqn:l2-ejgj}
\[\|S[F_{-m+\ell}, b_{j,m}]\|_{L^2(\R^2\times \mathcal{E})} \lesssim 2^{-\frac{\ell}2+j} [\assouadconst{\mathcal{E}}]^{\frac12} \|F_{-m+\ell}\|_{L^2(w_\gamma)},\quad \ell\ge 1.\]
On the other hand, by Proposition \ref{prop:l2-orth} and \eqref{eqn:l2-ejgj} we have
\[\| S[F_{-m+\ell},b_{j,m}] \|_{L^2(\R^2\times\mathcal{E})}\lesssim 2^{-(\gamma-\frac12)(-\ell)+j} [\assouadconst{\mathcal{E}}]^{\frac12} \|F_{-m+\ell}\|_{L^2(w_\gamma)},\ell\le 0. \]
From these estimates we get
\begin{align*}
&\Big(\sum_{0<m<j/2} \|S[F,b_{j,m}]\|_{L^2(\R^2\times\mathcal{E})}^2\Big)^{1/2}
\\&\le \sum_{\ell=-\infty}^\infty\Big(\sum_{0<m<j/2}
\|S[F_{-m+\ell},b_{j,m}]\|_{L^2(\R^2\times\mathcal{E})}^2\Big)^{1/2}
\\ &\lesssim 2^j [\assouadconst{\mathcal{E}}]^{\frac12}
\sum_{\ell=-\infty}^\infty
\min (2^{-\ell/2}, 2^{\ell(\gamma-\frac 12)} )
\Big( \sum_{0<m<j/2} \|F_{-m+\ell}\|_{L^2(w_\gamma)}^2 \Big)^{1/2}.
\end{align*}
By the disjointness of supports of the $F_k$, we have for each $\ell\in \mathbb Z$
\[\Big( \sum_{0<m<j/2} \|F_{-m+\ell}\|_{L^2(w_\gamma)}^2 \Big)^{1/2}\le \|F\|_{L^2(w_\gamma)}
\]
and since $\sum_{\ell=-\infty}^\infty
\min (2^{-\ell/2}, 2^{\ell(\gamma-\frac 12)})\lesssim (2\gamma-1)^{-1}$ for $1/2<\gamma\le 1$ we obtain \eqref{eqn:l2-weighted-impr}.
\qed

\subsection{Proof of Proposition \ref{prop:l2-orth}}
\label{sec:mainorth}
We begin by observing that we may assume without loss of generality that $\mathcal{E}$ is uniformly $2^{-j+2m}$--separated (in the sense of \eqref{eqn:unifsep}). This is because every uniformly $2^{-j}$--separated set $\mathcal{E}\subset (0,\infty)$ can be written as a disjoint union of at most
\[ 2 \sup_{|I|=2^{-j+2m}} \# (\mathcal{E}\cap I) \]
sets each of which is uniformly $2^{-j+2m}$--separated. \c

By duality, $\| S[F,b]\|_{L^2(\R^2\times\mathcal{E})}$ is equal to the supremum over all $G$ with $\|G\|_{L^2(\R^2\times\mathcal{E})}=1$ of
\[ \Big|\int_{\R^4} \widehat{F}(\xi-\eta,\eta) \Big[ \sum_{t\in \mathcal{E}} G(\xi,t) e^{\pm i t(|\xi-\eta|+|\eta|)} b(t,\xi,\eta) \Big]
\,d(\xi,\eta) \Big|. \]
By the Cauchy-Schwarz inequality applied in $(\xi,\eta)$ we estimate the previous by
\begin{equation}\label{eqn:orthpf-1}
\|F\|_{L^2(\R^4)} \Big( \int_{\R^4} \Big| \sum_{t\in \mathcal{E}} G(\xi,t) e^{\pm i t(|\xi-\eta|+|\eta|)} b(t,\xi,\eta) \Big|^2 d(\eta,\xi) \Big)^{1/2}
\end{equation}
For each fixed $\xi$ consider
\[
\int_{\R^2} \Big|\sum_{t\in \mathcal{E}} G(\xi,t) e^{i t(|\xi-\eta|+|\eta|)} b(t,\xi,\eta)\Big|^2 d\eta. \]
Passing to polar coordinates $\eta=\rho\theta$ with $\rho\ge 0$ and $\theta\in S^1\subset \R^2$ and fixing the angular variable $\theta$ we are left with the one-dimensional integral
\[
\int_0^\infty \Big|\sum_{t\in \mathcal{E}} G(\xi,t) e^{\pm i t(|\xi-\rho\theta|+\rho)} b(t,\xi,\rho\theta)\Big|^2 \rho d\rho. \]
Expanding the square we rewrite this integral as
\begin{equation}\label{eqn:orthpf-2}
\sum_{t,t'\in \mathcal{E}} G(\xi,t) \overline{G(\xi,t')} \Big[\int_0^\infty e^{\pm i (t-t')\tau(\rho)} \nu_{\xi,\theta}(t,t',\rho) d\rho \Big],
\end{equation}
where we have set
\[ \tau(\rho)=\tau_{\xi,\theta} (\rho) = |\xi-\rho\theta| + \rho, \]
\[ \nu_{\xi,\theta}(t,t',\rho) = b(t,\xi,\rho\theta)\overline{b(t',\xi,\rho\theta)} \rho. \]
Keep in mind that by the assumptions on $b$, on the support of $\nu_{\xi,\theta}$ we have $\rho\approx 2^j, |\xi-\rho\theta|\approx 2^j$ and the angle of $\theta$ with $\xi-\rho\theta$ is $\approx 2^{-m}$.
Observe that $\tau$ is strictly monotone increasing with
\begin{equation}\label{eqn:tauderiv1} \tau'(\rho) = 1 - \big\langle \theta, \frac{\xi-\rho\theta}{|\xi-\rho\theta|} \big\rangle= 1-\cos(\measuredangle(\theta, \xi-\rho\theta) ) \approx 2^{-2m}.
\end{equation}
Similarly, we will show that
\begin{equation}\label{eqn:tauderiv}
|\tau^{(N)}(\rho)| \lesssim_N 2^{-2m-(N-1)j}
\end{equation}
for every integer $N\ge 1$. In order to establish \eqref{eqn:tauderiv} we verify that
for each $N\ge 1$ there are coefficients $(a_{k,N})_{k=0,\dots,N}$ with $\sum_{k=0}^N a_{k,N}=0$ such that
\begin{equation}\label{eqn:induction-statement} \tau^{(N)}(\rho) = |v|^{-(N-1)} \sum_{k=0}^N a_{k,N} w^k,
\end{equation}
where $v=\xi-\rho\theta$ and $w=\tfrac{\langle v, \theta\rangle}{|v|}=\cos(\measuredangle(\theta,\xi-\rho\theta))$.
This claim implies \eqref{eqn:tauderiv} (note that the polynomial $\sum_{k=0}^N a_{k,N} w^k$ is divisible by $1-w$).
To prove the claim we use induction on $N$, with $a_{0,1}=1$, $a_{1,1}=-1$ by \eqref{eqn:tauderiv1}.

Calculate $\tfrac{d}{d\rho}|v|^{1-N}=(N-1)w|v|^{-N}$ and
$\tfrac{d}{d\rho} w^k=|v|^{-1}kw^{k-1}(w^2-1)$. Hence
\begin{multline*}|v|^N\frac{d}{d\rho} \Big(|v|^{1-N}\sum_{k=0}^Na_{k,N}w^k\Big)=
\sum_{k=0}^N(N-1+k) a_{k,N}w^{k+1}-\sum_{k=1}^N ka_{k,N}w^{k-1}\end{multline*}
which is
written as $\sum_{k=0}^{N+1} a_{k,N+1}w^k$, which is a polynomial of degree $N+1$, and one checks using the induction hypothesis that the sum of its coefficients are zero.
Hence \eqref{eqn:tauderiv} is verified.

As a consequence of \eqref{eqn:tauderiv},
\[ |\partial_\rho^N [ \tfrac1{\tau'} ] (\rho) | \lesssim_N 2^{2m-Nj}. \]
Moreover, since $b$ is strictly $(j,m)$-adapted and satisfies \eqref{eqn:l2-angular-symb},
\[ |\partial_\rho^N \nu_{\xi,\theta}(t,t',\rho)|\lesssim_N B^2\, 2^{j-Nj}\quad\text{for}\;N\le 2. \]
Integrating by parts twice then shows
\[ \Big| \int_0^\infty e^{\pm i (t-t')\tau(\rho)} \nu_{\xi,\theta}(t,t',\rho) d\rho \Big| \lesssim B^2\, 2^{2j} (1+2^{-2m+j} |t-t'|)^{-2}. \]
From this we may estimate \eqref{eqn:orthpf-2} by
\[ B^2\, 2^{2j} \sum_{t,t'\in\mathcal{E}} (1+2^{-2m+j} |t-t'|)^{-2} |G(\xi,t)G(\xi,t')|. \]
An application of the Cauchy--Schwarz inequality shows that the previous is dominated by
\[ B^2 2^{2j} \Big(\sum_{s\in \mathcal{E}-\mathcal{E}} (1+2^{-2m+j} |s|)^{-2}\Big) \Big(\sum_{t\in\mathcal{E}} |G(\xi,t)|^2\Big). \]
Since $\mathcal{E}$ is uniformly $2^{-j+2m}$--separated (see \eqref{eqn:unifsep}), the previous display is
\[ \lesssim B^2\, 2^{2j} \sum_{t\in \mathcal{E}} |G(\xi,t)|^2. \]
Hence we have proved that
\[ \Big( \int_{\R^4} \Big| \sum_{t\in \mathcal{E}} G(\xi,t) e^{\pm i t(|\xi-\eta|+|\eta|)} b(t,\xi,\eta) \Big|^2 d(\eta, \xi) \Big)^{1/2} \lesssim B\, 2^{j-\frac{m}2}, \]
recalling that $\|G\|_{L^2(\R^2\times\mathcal{E})}=1$. Note that the factor $2^{-\frac{m}2}$ stems from integration over the angular variable $\theta$.
In view of \eqref{eqn:orthpf-1} this concludes the proof of Proposition \ref{prop:l2-orth}. \qed

\subsection{Interlude: Exotic symbols}\label{sec:exotic} In the proof of Proposition
\ref{prop:l2-offdiag} we use estimates for oscillatory integrals which are equivalent to the $L^2$ results for pseudo-differential operators of symbol class $S_{\rho,\rho}^0$ considered by Calder\'on and Vaillancourt \cite{cv}; here we take $\rho=1/2$.
Consider symbols $a:\R^d\times \R^d\to \C$ such that there exists a constant $A>0$ with
\[ |\partial_x^\beta \partial_\xi^\alpha a(x,\xi)| \le A (1+|\xi|)^{-( |\alpha|-|\beta|)/2} \]
for multiindices $\alpha,\beta$ with, say, $|\alpha|\le 4d+1, |\beta|\le 4d+1$. We define the associated operator
\begin{equation}\label{eqn:pdodef}
P_a f(x) = \int_{\R^d} e^{i\langle x,\xi\rangle} a(x,\xi) f(\xi) d\xi,
\end{equation}
which is {\em a priori} defined at least on integrable functions.
Then one has the estimate
\[ \|P_a f\|_{L^2(\R^d)} \le c A \|f\|_{L^2(\R^d)}, \]
where $c$ is a constant only depending on the dimension $d$.
The proof (an application of the
Cotlar--Stein almost orthogonality lemma) is due to \cite{cv}, for an exposition see also
Stein \cite[Ch. VII, \S 2.5]{Steinbook3}.

\subsection{Proof of Proposition \ref{prop:l2-offdiag}}
\label{sec:offdiag}
Fix $t\in \mathcal{E}$. From the definition \eqref{eqn:def-S},
\begin{equation}\label{eqn:l2-offdiagpf-0}
S[F,b](\xi,t) = \int_{\R^4} e^{-i\langle \xi,y\rangle} F(y,z) \Big[\int_{\R^2} e^{i\Phi_{t,\xi,y-z}(\eta)} b(t,\xi,\eta)(\eta) d\eta \Big] \,d(y,z),
\end{equation}
where we have set
\[ \Phi_{t,\xi,w}(\eta) = \langle \eta,w\rangle \pm t(|\xi-\eta|+|\eta|). \]
Changing variables $z\mapsto y-w$ in \eqref{eqn:l2-offdiagpf-0} shows that $\|S[F,b](\cdot,t)\|_{L^2(\R^2)}^2$ equals
\begin{equation}\label{eqn:offdiagpf-1}
\int_{\R^2} \Big| \int_{\R^2}\int_{\R^2} e^{-i\langle \xi,y\rangle} F(y,y-w) \Big[ \int_{\R^2} e^{i\Phi_{t,\xi,w}(\eta)} b(t,\xi,\eta) d\eta\Big] dy \,dw\Big|^2 d\xi.
\end{equation}
Computing the gradient of the phase function with respect to $\eta$,
\begin{equation}\label{eqn:cv-ineq} \nabla \Phi_{t,\xi,w}(\eta) = w \pm t \Big( \frac{\eta}{|\eta|} - \frac{\xi-\eta}{|\xi-\eta|} \Big).
\end{equation}
The key observation is now that if $\xi$ and $\eta$ satisfy \eqref{eqn:l2-defadapted}, then
\begin{equation}\label{eqn:l2-offdiag-key}
t \Big| \frac{\eta}{|\eta|} - \frac{\xi-\eta}{|\xi-\eta|} \Big| \le 2\, \measuredangle(\eta,\xi-\eta)\le 2^{-m+6}.
\end{equation}
Therefore, if $|w|\ge 2^{-m+\ell+20}$ and $\ell\ge0$, then
\[ |\nabla \Phi_{t,\xi,w}(\eta)| \ge 2^{-m+\ell+19}. \]
This tells us that we should integrate by parts in the $\eta$--integral.
Define a first order differential operator acting on functions $a:\R^2\to\C$ by
\[ \mathcal{L}_{t,\xi,w} [a] = i\,\mathrm{div}\Big(a \frac{\nabla \Phi_{t,\xi,w}}{|\nabla \Phi_{t,\xi,w}|^2}\Big). \]
Integrating by parts we obtain
\[ \int_{\R^2} e^{i\Phi_{t,\xi,w}(\eta)} b(t,\xi,\eta) d\eta = \int_{\R^2} e^{i\Phi_{t,\xi,w}(\eta)} \mathcal{L}_{t,\xi,w}[ b(t,\xi,\cdot)](\eta) d\eta. \]
Plugging this back into \eqref{eqn:offdiagpf-1}, changing the order of integration and applying the Cauchy--Schwarz inequality in $\eta$ shows that \eqref{eqn:offdiagpf-1} is
\[\lesssim 2^{2j-m} \int_{\R^4} \Big| \int_{\R^4} e^{-i\langle \xi,y\rangle + i \langle \eta, w\rangle} F(y,y-w) \mathcal{L}_{t,\xi,w}[ b(t,\xi,\cdot)](\eta)\,d(y,w) \Big|^2 d(\eta,\xi).\]
Setting $G_\xi(w) = \int_{\R^2} e^{-i\langle \xi,y\rangle} F(y,y-w) dy$ we consider
\begin{equation}\label{eqn:offdiagpf-2}
\int_{\R^2} \Big| \int_{\R^2} e^{i \langle \eta, w\rangle} G_\xi(w) \mathcal{L}_{t,\xi,w}[b(t,\xi,\cdot)](\eta) dw \Big|^2 d\eta
\end{equation}
for each fixed $\xi$.
Recalling the support assumption on $F$, let $\psi$ be a smooth function on $\R^2$ that is equal to $1$ on $\{ |w|\ge 2^{20} \}$ and supported in $\{ |w|\ge 2^{19} \}$. Setting \[ a_{t,\xi}(w,\eta)=\mathcal{L}_{t,\xi,w}[b(t,\xi,\cdot)](\eta)\psi(2^{m-\ell} w), \]
we recognize \eqref{eqn:offdiagpf-2} as equal to $\|P_{a_{t,\xi}}^* G_\xi\|_{L^2(\R^2)}^2$ (see \eqref{eqn:pdodef}).
The product and chain rules show that $\mathcal{L}_{t,\xi,w}[a]$ equals \cjr\cas
\[ \frac{i\langle\nabla a, \nabla\Phi_{t,\xi,w}\rangle}{ |\nabla \Phi_{t,\xi,w}|^{2}} - \frac{2i a \langle D^2\Phi_{t,\xi,w} \nabla\Phi_{t,\xi,w}, \nabla\Phi_{t,\xi,w}\rangle}{|\nabla \Phi_{t,\xi,w}|^{4}}, \]
where $D^2\Phi_{t,\xi,w}$ denotes the Hessian matrix of $\Phi_{t,\xi,w}$. From this one can deduce the symbol estimate
\begin{equation}\label{eqn:offdiagpf-symb}
\big|\partial_{w}^\beta \partial_\eta^\alpha a_{t,\xi}(w,\eta) | \lesssim B\,2^{-(j-2m+\ell)} 2^{m|\beta|} 2^{-(j-m) |\alpha|}.
\end{equation}
Since $0\le m\le j/2$ and $a_{t,\xi}$ is supported in $\{(w,\eta):|\eta|\approx 2^j\}$, an application of the Calder\'on--Vaillancourt result in \S\ref{sec:exotic}
and $L^2$ duality yield that
\eqref{eqn:offdiagpf-2} is
\[ \|P_{a_{t,\xi}}^* G_\xi\|_{L^2(\R^2)}^2 \lesssim B\,2^{-2(j-2m+\ell)} \|G_\xi\|^2_{L^2(\R^2)}, \]
uniformly in $\xi$. Integrating over $\xi$ and making use of Plancherel's theorem we obtain the bound
\[ \|S[F,b](\cdot,t)\|_{L^2(\R^2)} \lesssim B\, 2^{j-\frac{m}2} 2^{-(j-2m+\ell)} \|F\|_{L^2(\R^2)}, \]
valid for each fixed $t$. This implies \eqref{eqn:l2-offdiag}. \qed

\section{A necessary condition}\label{sec:finiteunions}

We shall need a strengthening of a lower bound from \cite{ahrs} which had been stated there for Assouad regular sets.

\begin{lem} \label{lem:neclemm}
Let $E\subset [1,2]$ and $M_E:L^p\to L^q$. Then
\begin{enumerate}
\item[(i)] for all $\delta>0$, for all intervals $I$ with $\delta\le |I|\le 1$,
\begin{equation}\label{eqn:Q4lower}
N(E\cap I,\delta)^{1/q} \lc \|M_E\|_{L^p\to L^q}
\delta^{\frac 1p-\frac dq} \big (\delta/|I|\big)^{\frac{d-1}{2} (\frac 1p+\frac 1q-1)}.
\end{equation}
\item[(ii)] Let $Q_1$, $Q_{4,\gamma}$ be as in \eqref{eqn:Qdef}.
The point $(\tfrac1p,\tfrac1{pd})$ belongs to $\overline{\mathcal T_E}$ if and only if it belongs to the line segment $\overline{Q_1Q_{4,\gamma}}$ with $\gamma=\dim_\qA\!E$.\c
\end{enumerate}
\end{lem}

\begin{proof}
Fix an interval $I=[a,b]$ with $\delta<b-a <1$ and let $\rho=\delta/ |I|$.

Let $f_{\delta,\rho}$ be the characteristic function of a $\delta$-neighborhood of the spherical cap of diameter $\sqrt\rho$, specifically
\[ \{ (y',y_d)\,:\, ||y|-a|\le \delta,\,|y'|\le \sqrt\rho,\,y_d>0 \}. \]
Then
\[ \|f_{\delta,I}\|_p \approx (\delta \rho^{\frac{d-1}{2}})^{1/p} . \]
Choose a covering of $E\cap I$ by a collection $\mathcal{J}_I$ of pairwise disjoint half open intervals of length $\delta$ intersecting $E\cap I$.
Then $\# \mathcal{J_I}\ge N(E\cap I,\delta)$.

We now argue as in \cite{ahrs} and let $c\in (0,1)$ be a small constant, say $c<10^{-2}$. We shall verify that for all $t\in \cup_{J\in \mathcal J_I}J $ and $x=(x',x_d)$ such that $|x'|\le c \delta\rho^{-1/2}$ and
$|x_d+t-a|\le c \delta$,
\begin{equation}\label{eqn:counterex_lowerbd}
M_E f_{\delta,\rho}(x) \ge A_t f_{\delta,\rho}(x',x_d) \gtrsim \rho^{\frac{d-1}2}.
\end{equation}
Fix $y=(y',y_d)\in S^{d-1}$ with $|y'|\le c \sqrt\rho $. Then
\begin{align*}
|x+ty|^2 &= |x'|^2 + x_d^2 + 2t \inn{x'}{y'} + 2tx_d y_d + t^2\\
& = |x'|^2 + (x_d+t)^2 + 2tx_d (\sqrt{1-|y'|^2}-1) + 2t \inn{x'}{y'}.
\end{align*}
Since $|x_d + t - a|\le c \delta$ and $|x'|^2 \le c^2 \delta^2\rho^{-1}\le c^2 \delta$, we get
\[ ||x'|^2 + (x_d+t)^2 - a^2| \le 6 c \delta, \]
\[ |2t\inn{x'}{y'}| \le 4 |x'| |y'| \le 4 c^2 \delta, \]
and
\[ |2t x_d (\sqrt{1-|y'|^2}-1)| \le 2 (|t-a|+c\delta) |y'|^2 \le 2c (|I|+c\delta) \rho \le 4 c \delta.\]
Here we used that $|I|=\delta \rho^{-1}$. This implies
\[ ||x+ty|^2 - a^2| \le 14 c \delta, \]
and therefore $||x+ty|-a|\le \delta$ if $c$ is chosen small enough ($c=10^{-2}$ works).
Also, $|x'+ty'| \le |x'| + 2 |y'| \le \sqrt\rho$ so that $f_{\delta,\rho}(x+ty) = 1$. This proves \eqref{eqn:counterex_lowerbd}.

Since the intervals $J\in\mathcal{J}_I$ are disjoint, the corresponding regions of $x$ where \eqref{eqn:counterex_lowerbd} holds can be chosen disjoint, and therefore
\[ \|M_E f_{\delta,\rho} \|_q \gtrsim \rho^{\frac{d-1}2} \big [ \delta N(E\cap I,\delta) (\delta\rho^{-1/2})^{d-1}\big ]^{1/q}. \]
Thus we must have
\[\rho^{\frac{d-1}2} \big (N(E\cap I,\delta) \delta \cdot (\delta\rho^{-1/2})^{d-1}\big )^{1/q}
\lc \|M_E\|_{L^p\to L^q} (\delta\rho^{\frac{d-1}{2}})^{1/p}.\]
which yields \eqref{eqn:Q4lower}.

Regarding part (ii), by Theorem \ref{thm:upperbounds} the points $(\tfrac1p, \tfrac1{pd})$ belong to $\overline{\mathcal {T}_E}$ if they are on the line segment $\overline{Q_1Q_{4,\gamma}}$ with $\gamma=\dim_\qA\!E$. It follows from part (i) that this condition is also necessary.
\end{proof}

\begin{proof}[Proof of Theorem \ref{thm:finiteunions}]
Observe that \[\sup_{j=1,\dots,m} \|M_{E_j}\|_{p\to q}\le \|M_E\|_{p\to q} \le \sum_{j=1}^m \|M_{E_j}\|_{p\to q}. \]
By Theorem \ref{thm:upperbounds} we have
$\cap_{j=1}^m \mathcal{Q}(\beta_j,\gamma_j)\subset \overline{\mathcal{T}_E}.$
Also, since $E_j$ is $(\beta_j,\gamma_j)$-regular, the known necessary conditions ({\it cf.} \eqref{eqn:typeset-ar} and Lemma \ref{lem:neclemm} above) yield \[\mathcal{Q}(\beta_j, \gamma_j)^\complement \subset
\overline{\mathcal{T}_{E_j}}^\complement\subset\overline{\mathcal{T}_E}^\complement\]
for each $j=1,\dots,m$. Hence, $\overline{\mathcal{T}_E}\subset\cap_{j=1}^m \mathcal{Q}(\beta_j, \gamma_j)$.
\end{proof}

\section{Constructions of Assouad regular sets}\label{sec:assouadcon}
In the proof of Theorem \ref{thm:type} we rely on
a family of (quasi-)Assouad regular sets that is {\it uniform} in the following sense.

\begin{lem}\label{lem:uniform}
There exist sets $\{E_{\beta,\gamma}\,:\,0<\beta<\gamma\le 1\}$ such that each $E_{\beta,\gamma}$ is $(\beta,\gamma)$-Assouad regular and there exists $c\ge 1$ such that for all $0<\beta<\gamma\le 1$, $\delta\in (0,1)$ and intervals $I\subset [1,2]$ with $|I|>\delta$,
\begin{equation}\label{eqn:uniformestimate}
N(E_{\beta,\gamma}, \delta)\le c\,\delta^{-\beta},\quad N(E_{\beta,\gamma}\cap I,\delta)\le c\,\big(\tfrac{\delta}{|I|})^{-\gamma}.
\end{equation}
\end{lem}
To prove the lemma we shall modify a construction in \cite[\S 5]{ahrs}. Additional care is needed because of our requirement of uniformity. In what follows we fix $0<\beta<\gamma\le 1$.

\newcommand{\cantor}{\mathfrak{C}}
\subsection{Preliminary: Cantor set construction}
Let $J$ be a compact interval. For $0<\mu\le 1/2$ and an integer $m\ge 0$ we let $\cantor_{\mu,m}(J)$ denote the union of the $m$th generation
$\{ {I_{m,\nu}}\,:\,\nu=1,\dots,
2^m\}$ of intervals of length $\mu^m |J|$ that arise from starting with $J$ and repeatedly removing the open middle piece of length $(1-2\mu)|J|$. Then the set $\cantor_\mu(J)=\cap_{m\ge 0} \cantor_{\mu,m}(J)$ has Hausdorff, Minkowski and Assouad dimensions equal to $\gamma=-\log(2)/\log(\mu)$. Note $\cantor_{1/2,m}(J)=J$ for all $m\ge 0$.
For every $\delta\in (0,|J|)$ it can be seen that\c
\begin{equation*}
\tfrac12 |J|^\gamma \delta^{-\gamma}\le N(\cantor_\mu(J), \delta) \le 2 |J|^\gamma \delta^{-\gamma}
\end{equation*}
Similarly, for intervals $I\subset J$ and $\delta\in (0, |I|)$ one can verify\c
\begin{equation*}
\tfrac14 |I|^\gamma \delta^{-\gamma}\le N(\cantor_\mu(J)\cap I, \delta)\le 4 |I|^\gamma \delta^{-\gamma}.
\end{equation*}
However, to construct the Assouad regular sets we will not use the Cantor sets $\cantor_\mu(J)$ directly. Instead we will work with
\[ \mathfrak{C}^{\mathrm{mid}}_{\mu,m}(J) = \{ \text{midpoint of}\,I_{m,\nu}\,:\,\nu=1,\dots,2^m\}. \]
Write $\delta_m=\mu^m |J|$. Observe that for every $m\ge 0$, $\delta\in (0,|J|)$,
\begin{equation}\label{eqn:midmink}
N(\mathfrak{C}^{\mathrm{mid}}_{\mu,m}(J), \delta) \le \min(\delta_m^{-\gamma}, 2\delta^{-\gamma}) |J|^{\gamma}.
\end{equation}
If $\delta<\delta_m$ this holds with equality since $\mathfrak{C}^{\mathrm{mid}}_{\mu,m}(J)$ consists of $2^m=\delta_m^{-\gamma} |J|^\gamma$ points that are pairwise separated by $\ge \delta_m$. If $\delta\ge \delta_m$ then there exists $j\le m$ such that $\delta\in [\mu^j |J|, \mu^{j-1}|J|)$, so
\[ N(\mathfrak{C}^{\mathrm{mid}}_{\mu,m}(J), \delta) \le 2^{j} \le 2\delta^{-\gamma} |J|^\gamma\]
and \eqref{eqn:midmink} is proved.

We claim that \eqref{eqn:midmink} implies
\begin{equation}\label{eqn:midassou}
N(\mathfrak{C}^{\mathrm{mid}}_{\mu,m}(J)\cap I, \delta) \le \,
\begin{cases} 8\delta^{-\gamma} |I|^\gamma &\text{ if }\delta_m \le \delta \le |I|, \\
4\delta_m^{-\gamma} |I|^{\gamma} &\text{ if }\delta \le \delta_m \le |I|,\\
1 &\text{ if }|I|<\delta_m, \end{cases}
\end{equation}
i.e. $N(\mathfrak{C}^{\mathrm{mid}}_{\mu,m}(J)\cap I, \delta)\le \max(1, \min(8\delta^{-\gamma}, 4\delta_m^{-\gamma}) |I|^\gamma)$, valid for all $m\ge 0$, open subintervals $I\subset J$ with $|I|>\delta$.
To see this, first note that the inequality holds if $|I|<\delta_m$ (as the points in
$\mathfrak{C}^{\mathrm{mid}}_{\mu,m}(J)$ have mutual distance $\ge \delta_m$).
Let $|I|\ge \delta_m$. Let $\ell$ be the integer so that $|I|\in [\mu^\ell |J|, \mu^{\ell-1} |J|)$. Then $\ell\le m$.
Define
\[ \mathcal{V} = \{ \nu=1,\dots,2^{\ell-1}\,:\,I_{\ell-1,\nu}\cap I\not=\emptyset \}. \]
Observe that $\# \mathcal{V}\le 2$ (if $\mu<1/3$, then even $\#\mathcal{V}\le 1$). Next,
\[ \mathfrak{C}^{\mathrm{mid}}_{\mu,m}(J)\cap I \subset \bigcup_{\nu\in\mathcal{V}}
\mathfrak{C}^{\mathrm{mid}}_{\mu,m}(J)\cap I_{\ell-1,\nu} = \bigcup_{\nu\in\mathcal{V}} \mathfrak{C}^{\mathrm{mid}}_{m-\ell+1}(I_{\ell-1,\nu}). \]
Using \eqref{eqn:midmink}, $\mu^{m-\ell+1} |I_{\ell-1,\nu}| = \mu^m |J| = \delta_m$ and $|I_{\ell-1,\nu}|^\gamma\le 2 |I|^\gamma$,
\[ N(\mathfrak{C}^{\mathrm{mid}}_{m-\ell+1}(I_{\ell-1,\nu}), \delta) \le \min( 2\delta_m^{-\gamma}, 4\delta^{-\gamma} ) |I|^\gamma. \c\]
Since $\#\mathcal{V}\le 2$, this implies the claim.

\subsection{Assouad regular sets}\label{sec:assouadconsub}
Let $\lambda = 2^{-1/\beta}$ and $\mu=2^{-1/\gamma}$,
so that $\lambda<\mu\le 1/2$. Define
\[ J_k=[1+ \lambda^{k+1}, 1+\lambda^k],\quad \theta = 1 - \beta/\gamma,\quad m(k) = \lceil \tfrac{k}{\theta} \rceil. \]
We then set
$$F= \bigcup_{k=1}^\infty F_k, \text{ where } F_k=\cantor^{\mathrm{mid}}_{\mu,m(k)}(J_k).$$
The length of each the constituent intervals $I_{m(k),\nu}$, $\nu=1,\dots,2^{m(k)}$ of $\cantor_{m(k)}^\mu(J_k)$ is
\[ \sigma_k = |J_k| \mu^{m(k)}. \]
Since $\lambda<1/2$, $2^{-k/\beta-1}\le |J_k|=2^{-k/\beta}(1-\lambda) \le 2^{-k/\beta}$\c.
The choice of $m(k)$ is made so that $\sigma_k^\theta \approx 2^{-k/\beta}\approx |J_k|$. More precisely, \c
\begin{equation}\label{Jklen}
(2^{-1}\mu)^\theta 2^{-k/\beta} \le \sigma_k^\theta \le 2^{-k/\beta}.
\end{equation}

For open intervals $I\subset [1,2]$, $\delta\in (0,1)$, $|I|>\delta$ we claim that
\begin{equation}
\label{eqn:64assupper} N(F\cap I, \delta)\le 48\, (\delta/|I|)^{-\gamma}.
\end{equation}
This estimate
immediately yields $\dim_{\mathrm{A}}\!F\le \gamma$.

To prove \eqref{eqn:64assupper} first take $I\subset J_k$ open with $|I|>\delta$. Then by \eqref{eqn:midassou},
\begin{equation}\label{64:NEkI}
N(F_k\cap I,\delta)\le \max(1, 8 \min( \sigma_k^{-\gamma} , \delta^{-\gamma} ) |I|^{\gamma}).
\end{equation}
Then, for an arbitrary open interval $I\subset [1,2]$ and $\delta<|I|$,
\begin{align*}&N(F\cap I,\delta) \le \sum_{k\ge 0} N(F_k\cap (J_k\cap I),\delta) \\
& \le \sum_{\substack{k: |J_k|\ge |I|,\\ J_k\cap I\not=\emptyset}} N(F_k\cap (J_k\cap I),\delta) +
\sum_{k: |J_k|\le |I|} N(F_k\cap J_k,\delta). \end{align*}
By \eqref{64:NEkI} this is bounded by
\[\label{64:inbetween} 8 \sum_{\substack{k: |J_k|\ge |I|,\\ J_k\cap I\not=\emptyset}} \delta^{-\gamma} |I|^\gamma + 8 \sum_{k: |J_k|\le |I|} \delta^{-\gamma} |J_k|^\gamma \le 48\,\delta^{-\gamma} |I|^{\gamma} \]
which finishes the proof of \eqref{eqn:64assupper}.
Next we turn to $\overline{\dim}_{\mathrm{A},\theta}F$. We have \cjr
\begin{equation}\label{eqn:assouregpf1}
N(F\cap J_k,\sigma_k) = N(F_k, \sigma_k) = 2^{m(k)} = \mu^{-\gamma m(k)}= \sigma_k^{-\gamma} |J_k|^\gamma,
\end{equation}
which implies $\overline{\dim}_{\mathrm{A},\theta}\!F\ge \gamma$, because $|J_k|\approx \sigma_k^\theta$. Since $\overline{\dim}_{\mathrm{A},\theta}\!F\le \dim_{\mathrm{A}}\!F$ this proves
\[ \dim_{\mathrm{A}}\!F=\overline{\dim}_{\mathrm{A},\theta}\!F=\gamma.\]

Regarding $\dim_{\mathrm{M}}F$ we see that because of
$\sigma_k^{-\gamma} |J_k|^\gamma \approx \sigma_k^{-\beta}$ one gets $\dim_{\mathrm{M}}\!F\ge \beta$.
Finally, for every $\delta\in (0,1)$,
\begin{equation}\label{64:mink}
N(F, \delta) \le c_{F} \delta^{-\beta},
\end{equation}
where the constant $c_{F}\ge 1$ depends only on $\beta,\gamma$ (and may blow up as $\beta,\gamma$ tend to $0$). This follows since
\[ N(F,\delta) \le N(\cup_{k:\delta\ge |J_k|} F_k, \delta) + \sum_{k:\sigma_k<\delta<|J_k|} N(F_k,\delta) + \sum_{k:\sigma_k\ge \delta} N(F_k,\delta), \]
which by choice of $m(k)$ and \eqref{64:NEkI} is
\[ \lesssim_{\beta,\gamma} 1+\sum_{k\,:\,\sigma_k<\delta<\sigma_k^{\theta}} \delta^{-\gamma} \sigma_k^{\gamma-\beta} + \sum_{k\ge 0: \sigma_k\ge \delta} \sigma_k^{-\gamma} \sigma_k^{\gamma-\beta} \lesssim_{\beta,\gamma} \delta^{-\beta}. \]
To resolve the blowup in \eqref{64:mink} we use an affine transformation. Define
\[ E_{\beta,\gamma} = 1 + c_{F}^{-1/\beta} F\subset [1,2].\]
Then by \eqref{64:mink}, $N(E_{\beta,\gamma}, \delta)=N(F, c_F^{1/\beta} \delta) \le c_F (c_F^{1/\beta} \delta)^{-\beta} =\delta^{-\beta}$ for all $\delta<c_F^{-1/\beta}$ and $N(E_{\beta,\gamma}, \delta)\le 1\le \delta^{-\beta}$ for all $\delta\in [c_F^{-1/\beta}, 1)$. Similarly, by \eqref{eqn:64assupper}, $N(E_{\beta,\gamma}\cap I,\delta)\le c\, (\delta/|I|)^{-\gamma}$ for all $\delta\in (0,1)$ and intervals $I\subset [1,2]$ with $|I|>\delta$. This concludes the proof of Lemma \ref{lem:uniform}. \qed

\subsection{Zero Minkowski and positive Assouad dimension}\label{sec:minkzero}
We consider $E\subset [1,2]$ with $\dim_{\mathrm{M}}\!E=0$.
Note that $\overline{\dim}_{\mathrm{A},\theta}\!E=0$ for all $\theta<1$ (\cite[Cor.\,3.3]{fraser-yu1}),
hence also $\dim_{\qA}\!E=0$. However (as already remarked in \cite{fhhty}) the Assouad dimension can be any given $\gamma\in [0,1]$. We prove this for the sake of completeness.

\begin{lem}\label{lem:minkzero}
For every $\alpha\in [0,1]$ there exists a set $G_{\alpha}\subset [1,2]$ such that $\dim_{\mathrm{M}}\!G_{\alpha}=0$ and $\dim_{\mathrm{A}}\!G_{\a}=\a$.
\end{lem}
\begin{proof}

If $\a=0$ let $G_0=\{\text{point}\}$. For $\a=1$ let $G_1= \{1+2^{-\sqrt{\ell}}: \ell\ge 1\}$.

It remains to consider the case $0<\a<1$. Let
\[ J_n=[1+2^{-2^n-1}, 1+2^{-2^n}],\quad F_n=\mathfrak C^{\text{mid}}_{\mu,n} (J_n), \quad \mu=2^{1/\a}.\]
Next, let $n_0=n_0(\a)$ be the smallest non-negative integer so that $\a 2^{n_0} \ge 1$ and define
\[ G_\a = \bigcup_{n\ge n_0} F_n \subset [1,2]. \]

We first consider the Minkowski dimension. Pick $\delta\in (0,1)$ and observe that
$\cup_{n: 2^{-2^n}\le \delta}F_n$
is covered by $[0,\delta]$ and that the cardinality of
$\cup_{n: 2^{-2^n}>\delta}F_n$ is bounded by $\sum_{n:2^{-2^n}> \delta} 2^n\le 2 \log_2(\delta^{-1})$. Hence $\dim_{\mathrm{M}} G_\a=0$.

It remains to show that $\dim_{\mathrm{A}} G_\a=\a$.
Each of the intervals that make up $\cantor_{\mu,n}(J_n)$ has length
\[ \sigma_n = \mu^n |J_n| = \mu^n 2^{-2^n-1}. \]
Note by definition that
\[N(F_n\cap J_n, \sigma_n )= N(F_n,\sigma_n)= 2^n = \sigma_n^{-\a} |J_n|^\a.\]
Thus we get $\dim_{\mathrm{A}} G_\a\ge \a$. It remains to show $\dim_{\mathrm{A}} G_\a\le \a$. Here we proceed as in \S \ref{sec:assouadconsub}. If $I\subset J_n$ with $|I|>\delta$ then by \eqref{eqn:midassou},
\[ N(F_n \cap I, \delta) \le 8 \delta^{-\a}|I|^\a. \]
This implies for an arbitrary open interval $I\subset [1,2]$ with $|I|>\delta$ that
\[ N(G_\a\cap I, \delta) \le 8 \sum_{\substack{n:|J_n|\ge |I|\\J_n\cap I\not=\emptyset}} \delta^{-\a} |I|^\a + 8 \delta^{-\a} \sum_{\substack{n\ge n_0:|J_n|\le |I|}} |J_n|^\a \le 48\, \delta^{-\a} |I|^\a. \]
This is because the first sum in this display has at most two terms while
\begin{align*} \sum_{\substack{n\ge n_0:\\|J_n|\le |I|}} |J_n|^\a &\le \sum_{\substack{n\ge n_0: \\2^{-2^n-1}\le |I|} }(2^{-2^n - 1})^\a \le 2^{-\a} \sum_{n\ge 0: 2^{-2^{n}}\le (2|I|)^{2^{-n_0}}} 2^{-2^n} \\
& \le 2^{-\a} \sum_{\substack{n\ge 0:\\ 2^{-n} \le (2 |I|)^{2^{-n_0}}}} 2^{-n} \le 2^{-\a+1} (2 |I|)^{2^{-n_0}} \le 4 |I|^{\a}. \qedhere\end{align*}
\end{proof}

\section{Convex sets occurring as closures of type sets}\label{sec:type}
In this section we prove Theorem \ref{thm:type}.
Let $\W\subset [0,1]^2$ be a closed convex set and let $0\le \beta\le\gamma\le 1$ be such that \begin{equation}\label{eqn:WTEincl}
\mathcal{Q}(\beta,\gamma)\subset \W\subset \mathcal{Q}(\beta, \beta)
\end{equation}
and suppose that $\gamma$ is minimal with this property. If $\overline{\mathcal T_E}=\W$ for some $E\subset [1,2]$ then we have $\beta=\dim_{\mathrm M}\!E$, moreover it follows from part (ii) of Lemma \ref{lem:neclemm} that $\gamma=\dim_\qA\!E$.
In what follows it thus suffices to prove the existence of $E$ satisfying
$\overline{\mathcal T_E}=\mathcal{W}.$

If $\beta=\gamma$, then we may take $E$ to be a self similar Cantor set of Minkowski dimension $\beta$. If $\beta=0$, then $\W=\mathcal{Q}(0,0)$, so the single average example $E=\{\mathrm{point}\}$ works. It remains to consider the case $0<\beta<\gamma\le 1$.

We may also assume that
$\mathcal{Q}(\beta,\gamma)\subsetneq \W\subsetneq \mathcal{Q}(\beta,\beta)$ (if $\W=\mathcal{Q}(\beta,\gamma)$ then we choose any $(\beta,\gamma)$-regular set for $E$, such as the one from Lemma \ref{lem:uniform}, and if $\W=\mathcal{Q}(\beta,\beta)$ we again choose a Cantor set). \cmt{That is maybe not necessary}

Let $\mathfrak{L}$ denote the set of lines that pass through at least one point in $\partial \W$ but are disjoint from the interior of $\W$. Each line $\ell\in\mathfrak{L}$ divides the plane into two half-spaces. We denote by $\mathfrak{H}(\ell)$ the closed half-space that contains $\W$. Then
\begin{equation}\label{eqn:halfspaceintersec}
\W = \bigcap_{\ell\in\mathfrak{L}} \mathfrak{H}(\ell).
\end{equation}
There exists a countable subset $\mathfrak{L}'\subset \mathfrak{L}$ such that $\W=\cap_{\ell\in\mathfrak{L}'} \mathfrak{H}(\ell)$. We further select a subset $\mathfrak{L}^\flat\subset \mathfrak{L}'$ consisting only of those lines that do not contain any of the edges of $\mathcal{Q}(\beta,\gamma)$.
$\mathfrak{L}^\flat$ must be non-empty because $\W\supsetneq \mathcal{Q}(\beta,\gamma)$.

Since $\mathfrak{L}^\flat$ is countable we may write $\mathfrak{L}^\flat = \{\ell_1,\ell_2,\dots\}$.
The line $\ell_n$ intersects the line segment connecting $Q_{3,\beta}$ and $Q_{3,0}$ in a point $Q_{3,\beta_n}$ for some $\beta_n\in [0,\beta]$. The line $\ell_n$ also intersects the line segment connecting $Q_{4,\gamma}$ and $Q_{4,\beta}$ in a point $Q_{4,\gamma_n}$ for some $\gamma_n\in [\beta, \gamma]$. This is illustrated in Figure \ref{fig:tangents}.

\begin{figure}[ht]
\begin{tikzpicture}[scale=100]
\def\d{4}
\def\b{.3}
\def\g{1}

\def\clipStyle{solid}
\def\clipLineOpacity{1}
\def\clipFillOpacity{0}

\begin{scope}[yscale=1]
\definecoords

\def\ptsize{.015pt}

\def\clipLx{ \Qfourx{\g} - .005 }
\def\clipLy{ \Qfoury{\g} - .006 }
\def\clipHx{ \Qthreex{0} + .01 }
\def\clipHy{ \Qthreey{\b} + .005 }
\def\clipCx{ (\clipLx+\clipHx)/2 }
\def\clipCy{ (\clipLy+\clipHy)/2 }
\def\cliprad{ .2 }

\def\theGamma{\gamma}

\coordinate (clipLL) at ({\clipLx}, {\clipLy});
\coordinate (clipHH) at ({\clipHx}, {\clipHy});
\draw [\clipStyle, opacity=\clipLineOpacity] (clipLL) rectangle (clipHH);
\fill [opacity=\clipFillOpacity] (clipLL) rectangle (clipHH);
\clip (clipLL) rectangle (clipHH);

\drawauxlines{Q4b}{Q2}

\def\QbgDrawCritSeg{1}
\def\QbgCritSegStyle{solid}
\def\QbgCritSegOpacity{.2}
\def\AFillColor{black}
\def\QbbFillOpacity{0}
\def\QbgFillOpacity{.1}
\def\AFillOpacity{.1}

\drawQbg
\drawQbb
\lblQbb
\drawA
\lblQthreezero

\def\tangentOpacity{.8}
\newcommand{\drawATangent}[3]{
\path (Q4g) .. controls (C1) and (C2) .. (Q3) node [sloped,pos=#1,minimum width=10cm](tang) {};
\coordinate (tmpA) at (intersection of tang.west--tang.east and Q4g--Q30);
\coordinate (tmpB) at (intersection of tang.west--tang.east and Q3--Q30);
\draw [opacity=\tangentOpacity] (tmpA) -- (tmpB);
\fill (tmpA) node [below right] {#2} circle [radius=\ptsize];
\fill (tmpB) node [right] {#3} circle [radius=\ptsize];
}

\def\ptsize{.01pt}
\foreach \x [evaluate=\x as \p using \x*0.1] in {2,...,9} {
\ifnum\x=5
\drawATangent{\p}{$Q_{4,\gamma_n}$}{$Q_{3,\beta_n}$}
\else
\drawATangent{\p}{}{}
\fi
}

\end{scope}
\end{tikzpicture}
\caption{$d=4$, $\beta=0.3$, $\gamma=1$}\label{fig:tangents}
\end{figure}
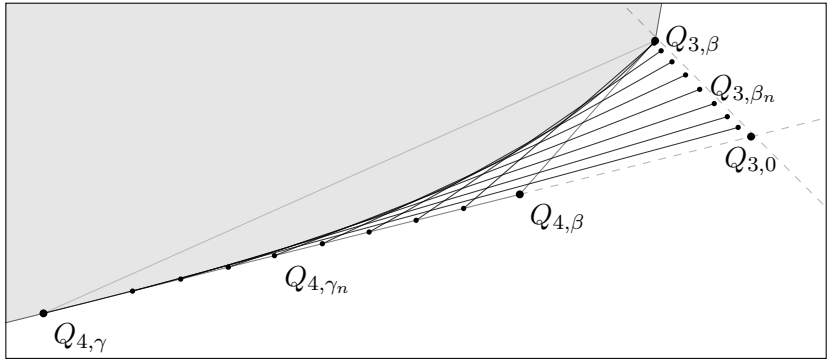
Then from \eqref{eqn:halfspaceintersec},
\begin{equation}\label{eqn:Qintersec}
\W = \bigcap_{n\ge 1} \mathcal{Q}(\beta_n, \gamma_n).
\end{equation}
Let $E_{\beta,\gamma}\subset [1,2]$ be as in Lemma \ref{lem:uniform}.
Let $L:\mathbb{N}\to\mathbb{N}$
be strictly increasing (we may just choose $L(n)=n$ in this proof but we will need to
make a different choice to prove Remark \ref{rem:assouad}).
Define
\begin{equation} \label{eqn:En-constr}
\begin{aligned}E_n = 1 + 2^{-L(n)-1} &E_{\beta_n,\gamma_n} \subset [1+2^{-L(n)-1},
1+2^{-L(n)}],\quad
\\& E = \bigcup_{n=1}^\infty E_n\subset [1,2].
\end{aligned}\end{equation}
It remains to show that $\overline{\mathcal{T}_E} = \W$. First we have $\|M_E\|_{p\to q} \ge \|M_{E_n}\|_{p\to q}$ for every $n\ge 1$. Since $E_n$ is $(\beta_n,\gamma_n)$-regular,
we have $\overline{\mathcal{T}_E} \subset \overline{\mathcal {T}_{E_n}}=\mathcal{Q}(\beta_n,\gamma_n)$, {\it cf}. \eqref {eqn:typeset-ar}. Hence $\overline{\mathcal{T}_E}\subset \W$ by \eqref{eqn:Qintersec}.

To prove $\W\subset \overline{\mathcal{T}_E}$ let us take a point $(p_*^{-1},q_*^{-1})$ in the interior of $\W$. It suffices to show that $M_E$ is bounded $L^{p_*}\to L^{q_*}$.
By the construction of $E_n$ and \eqref{eqn:uniformestimate},
\begin{equation}\label{eqn:uniformestimate2}
N(E_n, \delta) \le c \delta^{-\beta_n},\quad N(E_n\cap I,\delta)\le c \big(\tfrac{\delta}{|I|}\big)^{-\gamma_n}
\end{equation}
for every $n\ge 1, \delta\in (0,2^{-n})$ and every interval $I\subset [1,2]$ with $|I|>\delta$.
Here it is important that $c$ does not depend on $n$.

Crucially, for each $j$ all $E_n$ with $n>j$ are contained
in $[1, 1+2^{-L(j)}]$ and thus
in $[1, 1+2^{-j}]$. Therefore
we can estimate \c
\begin{equation}\label{eqn:5_basic}
\big\|\sup_{t\in E} |\A^j_t f| \big\|_{q_*} \le \sum_{n\le j} \big\|\sup_{t\in E_n} |\A^j_t f|\big\|_{q_*} +
\big\|\sup_{1\le t\le 1+2^{-j} } |\A_t^j f|\big\|_{q_*}.
\end{equation}
The second term on the right hand side is dominated by
\[\|\A_1 f\|_{q_*} +\int_0^{2^{-j} }\big\|\tfrac{d}{dt} \A_t^j f\big \|_{q_*}dt\lesssim 2^{-a j}\|f\|_{p_*}
\]for some $a>0$ because $(p_*^{-1}, q_*^{-1})$ is in the interior of $\mathcal{Q}(0,0)\supset \W$.

It remains to estimate the first term in \eqref{eqn:5_basic}. Since $(p_*^{-1}, q_*^{-1})$ is in the interior of $\W$, it is away from the boundary of each $\mathcal{Q}(\beta_n,\gamma_n)$ by a positive distance independent of $n$.
By Corollary \ref{cor:interior} we now obtain $\varepsilon =\varepsilon(p_*,q_*) >0$ not depending on $n$ such that \cas
\begin{equation}\label{eqn:5_decay}
\big\|\sup_{t\in E_n} |\A^j_t f|\big \|_{q_*} \lesssim_{d,p_*,q_*} 2^{-j \varepsilon(p_*,q_*)} \|f\|_{p_*}
\end{equation}
for all $j\ge 0$ and $n\ge 1$. To see that the implicit constant does not depend on $n$ one uses \eqref{eqn:Ajmain} and that \eqref{eqn:uniformestimate2} implies \cas
\[ [\chi^{E_n}_{\mathrm{A},\gamma_n}(2^{-j})]^{b_1} [\chi_{\mathrm{M},\beta_n}^{E_n}(2^{-j})]^{b_2} \le c \]
with $c$ depending only on the dimension $d$.
\cmt{\jr{changed}}
Hence, the first term on the right hand side of \eqref{eqn:5_basic} is $\lesssim j 2^{-j\varepsilon} \|f\|_{p_*}$. This concludes the proof that $(p_*^{-1}, q_*^{-1})\in \mathcal{T}_E$.

\subsection*{Proof of Remark \ref{rem:assouad}}
Without loss of generality let $\gamma_*=\gamma$.

We need to make a judicious choice of the sequence $L(n)$ to achieve
$\dim_{\mathrm{A}}\!E=\gamma$.
Define $L(n)$ iteratively such that for
$n=2,3,\dots$
\begin{equation}\label{eqn:Ln}L(n)\ge \max
\{L(k)+\gamma_n^{-1}(n-k): \,\, 1\le k\le n-1\}.\end{equation}
We then claim that the construction actually yields
\begin{equation}\label{eqn:assouadabove}
\dim_{\mathrm{A}}\!E \le \gamma.
\end{equation} and since $\gamma=\dim_{\qA}\!E\le \dim_{\mathrm A}\!E$ we actually get equality in \eqref{eqn:assouadabove}.

We now show \eqref{eqn:assouadabove}.
Let $J_n=[1+2^{-L(n)-1}, 1+2^{-L(n)}]$. Let $\delta<|I|\le 1$ and observe that
\[
\sum_{\substack{n:|J_n|\ge |I|,\\ J_n\cap I\not=\emptyset}} N(E_{n}\cap I,\delta) \le c \sum_{\substack{n:|J_n|\ge |I|,\\ J_n\cap I\not=\emptyset}} (\delta/|I|)^{-\gamma_{n}} \le 2 c (\delta/|I|)^{-\gamma}
\]
using that $\gamma_n\le \gamma$ and there are at most two $n$ such that $J_n$ intersects $I$ and $|J_n|\ge |I|$.
Next
let $n_\circ$ be the smallest $n$ for which $|J_n|<|I|$. By \eqref{eqn:Ln},
$L(n)\gamma_n-L(n_\circ)\gamma_{n}\ge n-n_\circ$. We obtain \c
\begin{align*}
&\sum_{n: |J_n|< |I|} N(E_n\cap (J_n\cap I),\delta)
\le c \sum_{n:|J_n|<|I|} (\delta/|J_n|)^{-\gamma_n}
\\&\quad
= c \sum_{n:|J_n|<|I|} (\delta/|I|)^{-\gamma_n} (|I|/|J_n|)^{-\gamma_n}\le c(\delta/|I|)^{-\gamma} \sum_{n\ge n_\circ} (|I|/|J_n|)^{-\gamma_n},
\end{align*}
and
\begin{align*}&\sum_{n\ge n_\circ} (|I|/|J_n|)^{-\gamma_n} = \sum_{n\ge n_\circ}2^{-(L(n)+1)\gamma_n} |I|^{-\gamma_n}
\\&\le
\sum_{n\ge n_\circ}2^{-L(n)\gamma_n} 2^{L(n_\circ)\gamma_n}
\le \sum_{n\ge n_\circ} 2^{-(n-n_\circ)} \le 2.\end{align*}
Hence $\sum_{n: |J_n|< |I|} N(E_n\cap (J_n\cap I),\delta) \lesssim 2c (\delta/|I|)^{-\gamma}.$
The two cases imply \eqref{eqn:assouadabove}, and this settles the case $\gamma=\gamma_*$ in Remark \ref{rem:assouad}.

Finally consider the case $\gamma_*\in (\gamma,1]$. Take a set $E'$ as constructed above with
$\dim_{\mathrm M}\! E'=\beta$, $\dim_{\mathrm{qA}}\!E' =\dim_{\mathrm A}\!E'=\gamma$ and $\overline{\mathcal T_{E'}}=\W$ and a set
$G_{\gamma_*} $ as in Lemma \ref{lem:minkzero} with $\dim_{\mathrm M}G_{\gamma_*}=0 $ and $\dim_{\mathrm A}G_{\gamma_*}=\gamma_* $. Define
$E=E'\cup G_{\gamma_*}$. Then
$\dim_{\mathrm{A}}\!E=\max(\dim_{\mathrm{A}}\! E', \dim_{\mathrm{A}}\! G_{\gamma_*})=\gamma_*$ and
since $\overline{\mathcal T_{ G_{\gamma_*}}}= \mathcal Q(0,0)\supset\W$ we see that $\overline{\mathcal{T}_{E}}=\W$.
\qed

\end{document}